\documentclass[a4paper]{article}

\usepackage{graphicx}
\usepackage{bbm}
\usepackage{amsmath,amssymb,amsthm}
\usepackage{a4wide}

\newcommand{\Expectation}{{\ensuremath{\mathbf{E}}}}
\providecommand{\coreGent}{\SetH}

 \providecommand{\C}{{\ensuremath{\mathbf{C}}}}

 \providecommand{\SetH}{{\ensuremath{\mathbbm{H}}}}

 \renewcommand{\P}{{\ensuremath{\mathbf{P}}}}
 
 \providecommand{\P}{{\ensuremath{\mathbf{P}}}}
 
 \providecommand{\R}{{\ensuremath{\mathbbm{R}}}}

 \providecommand{\Z}{{\ensuremath{\mathbbm{Z}}}}
 \providecommand{\1}{{\ensuremath{\mathbbm{1}}}}

 \providecommand{\At}    {{\ensuremath{\tilde{A}}}}

 \providecommand{\Gt}    {{\ensuremath{\tilde{G}}}}

 \providecommand{\Lt}    {{\ensuremath{\tilde{L}}}}

 \providecommand{\St}    {{\ensuremath{\tilde{S}}}}

 \providecommand{\Yt}    {{\ensuremath{\tilde{Y}}}}

 \providecommand{\Lh}    {{\ensuremath{\hat{L}}}}

 \providecommand{\mt} {{\ensuremath{\tilde{m}}}}

 \providecommand{\pt} {{\ensuremath{\tilde{p}}}}

 \providecommand{\Cb}    {{\ensuremath{\bar{C}}}}

 \providecommand{\Lb}    {{\ensuremath{\bar{L}}}}

 \providecommand{\Sb}    {{\ensuremath{\bar{S}}}}

 \providecommand{\pb} {{\ensuremath{\bar{p}}}}

\providecommand{\mut}  {{\ensuremath{\tilde{\mu}}}}

\providecommand{\pit}  {{\ensuremath{\tilde{\pi}}}}
\providecommand{\pih}  {{\ensuremath{\hat{\pi}}}}

\providecommand{\pib}  {{\ensuremath{\bar{\pi}}}}

\providecommand{\qed}{\hfill\mbox{$\Box $}}

\providecommand{\ro}[1]    {{{(#1)}}}
\providecommand{\rob}[1]   {{{\bigl(#1\bigr)}}}
\providecommand{\robb}[1]  {{{\biggl(#1\biggr)}}}
\providecommand{\roB}[1]   {{{\Bigl(#1\Bigr)}}}

\providecommand{\rub}[1]   {{{\bigl(#1\bigr)}}}

\providecommand{\ruB}[1]   {{{\Bigl(#1\Bigr)}}}

\providecommand{\eckb}[1]  {{{\bigl[#1\bigr]}}}
\providecommand{\eckbb}[1] {{{\biggl[#1\biggr]}}}
\providecommand{\eckB}[1]  {{{\Bigl[#1\Bigr]}}}

\providecommand{\curlb}[1]  {{{\bigl\{#1\bigr\}}}}

\providecommand{\abs}[1]  {{\ensuremath{|#1|}}}

\providecommand{\al}      {{\ensuremath{\alpha}}}

\providecommand{\ld}      {{\ensuremath{\lambda}}}

\providecommand{\eps}     {{\ensuremath{\varepsilon}}}
\providecommand{\dl}      {{\ensuremath{\delta}}}

\providecommand{\limepsO}{{\ensuremath{{\displaystyle \lim_{\eps \ra 0}}}}}

\providecommand{\limtO}  {{\ensuremath{{\displaystyle \lim_{t \ra 0}}}}}

\providecommand{\qqastO} {\ensuremath{\qquad\text{as }t\to0}}

\providecommand{\ra}{\rightarrow}

\providecommand{\lradlO}  {\xrightarrow{\dl  \ra0}}

\providecommand{\lratO}{\xrightarrow{t  \ra0}}

\providecommand{\wlimeps}%
      {{\ensuremath{\stackrel{\eps \rightarrow \infty}%
                            {\Longrightarrow}}}}

\providecommand{\mal}{\ensuremath{{\displaystyle \cdot}}}

\providecommand{\fa}{\ensuremath{\;\;\forall\;}}

\providecommand{\Dom}{{\ensuremath{\mathfrak{D}}}}

\providecommand{\Law}[2][]{{\ensuremath{\mathcal L^{#1}\left(#2\right)}}}

\providecommand{\uline}[1]{{\ensuremath{\underline{#1}}}}

\providecommand{\eqd}{{\ensuremath{\;\overset{d}{=}\;}}}

\newcommand{\cadlag}{c\`adl\`ag}
\newcommand{\Nat}{\ensuremath{\mathbbm{N}}}

\newtheorem{lem}    {Lemma}[section]
\newtheorem{rem}    {Remark}[section]
\newtheorem{thm}    {Theorem}[section]
\newtheorem{assumpt}{Assumption}[section]
\newcommand{\acks}{\section*{Acknowledgements}}
\begin{document}

\author{Martin Hutzenthaler\footnote{Research supported by the DFG in the Dutch German Bilateral Research Group "Ma\-the\-ma\-tics of Random Spatial Models from Physics and Biology" (FOR 498)} \footnote{Research supported by EPSRC Grant no GR/T19537/01}
\\
{\it Goethe-University Frankfurt}
\and
Jesse E. Taylor\\
{\it University of Oxford}
}

\title{\bf Time Reversal of Some Stationary Jump-Diffusion Processes\\ from Population Genetics} 
\date{}
\maketitle
\def\thefootnote{\fnsymbol{footnote}}
\def\@makefnmark{\hbox to\z@{$\m@th^{\@thefnmark}$\hss}}
\footnotesize\rm\noindent
\footnote[0]{{\it AMS\/\ 2010 subject
classifications: {\rm Primary 60J60; secondary 60J55, 92D10}}. }
\normalsize\rm

\begin{abstract}
We describe the processes obtained by time reversal of a class of stationary 
jump-diffusion processes that model the dynamics of genetic variation in populations 
subject to repeated bottlenecks.  Assuming that only one lineage survives each 
bottleneck, the forward process is a diffusion on $[0,1]$ that jumps to the boundary 
before diffusing back into the interior.  We show that the behavior of the time-reversed 
process depends on whether the boundaries are accessible to the diffusive motion 
of the forward process.  If a boundary point is inaccessible to the forward diffusion, 
then time reversal leads to a jump-diffusion that jumps immediately into the interior 
whenever it arrives at that point.  If, instead, a boundary point is accessible, then the 
jumps off of that point are governed by a weighted local time of the time-reversed 
process.
\end{abstract}

%
%
\section{Introduction} 

Kingman's observation that the genealogy of a random sample of individuals from a 
panmictic, neutrally-evolving population can be represented as a Markov process
\cite{Ki82a,Ki82b} ranks as one of the most influential contributions of mathematical 
population genetics.  Not only has the coalescent led to a deeper understanding of 
evolution in neutral populations, but it also plays a central role in statistical genetics 
where it facilitates the efficient simulation of sample genealogies.  Unfortunately, 
the Markov property that makes Kingman's coalescent both mathematically and 
computationally tractable is usually not shared by genealogical processes in 
populations composed of non-exchangeable individuals.  In particular, this is true 
when there are fitness differences between individuals, since then the selective 
interactions between individuals cause genealogies to depend on the history of lineages 
that are non-ancestral to the sample.  The key to overcoming this difficulty is to extend 
the genealogy to a higher-dimensional process that does satisfy the Markov property.  
This has been done in two ways.  One approach is to embed the genealogical tree 
within a graphical process called the ancestral selection graph \cite{KN97,NK97,DK99AAP}
in which lineages can both branch and coalesce.   The intuition behind this construction 
is that the effects of selection on the genealogy can be accounted for by keeping track 
of a pool of potential ancestors which includes lineages that have failed to persist due 
to being out-competed by individuals of higher fitness.

An alternative approach was proposed by Kaplan et al.\ (1988) \cite{kdh88}, who showed 
that the genealogical history of a sample of genes under selection can be represented
as a structured coalescent process.  Here we think of the population as being divided 
into several panmictic subpopulations (called genetic backgrounds) which consist of individuals 
that share the same genotype at the selected locus.  Because individuals with the same 
genotype are exchangeable (i.e., they have the same fitness), the rate of coalescence 
within a background depends only on the size of the background and the number of 
ancestral lineages sharing that genotype.  Thus, to obtain a Markov process, we need 
to keep track of two kinds of information: (i) the types of the ancestral lineages, and (ii) 
the frequencies of the alleles segregating at the selected locus, followed backwards in 
time.  For many applications it is assumed that the population is at equilibrium and that 
the forwards in time dynamics of the allele frequencies are described by a stationary 
diffusion process.  In this case, the ancestral process of allele frequencies can be 
identified by time reversal of the diffusion process.  In particular, if the diffusion process 
is one-dimensional, then the time-reversed process conveniently has the same law 
as the forward process.  A formal derivation of the structured coalescent process for 
such an equilibrium population is given in \cite{BES04} and various applications 
are discussed in \cite{BE04,CG04,Tay07}.

The focus of this article is on the time reversal of a population genetical model that 
incorporates mutation, selection, genetic drift and population bottlenecks.  To be concrete, 
consider a locus with two alleles, $A_{0}$ and $A_{1}$, and let $p^{N}(t)$ denote the 
frequency of $A_{1}$ at time $t$ in a population of size $N$.  In the absence of bottlenecks, 
we will suppose that the jump process $p^{N}(\cdot)$ can be approximated by the
Wright-Fisher diffusion $p(\cdot)$ with generator
\begin{equation}  \begin{split}  \label{eq:A}
	A \phi(p) & =  \frac{1}{2} p(1-p) \phi''(p) + (\mu_{0} (1-p) - \mu_{1} p + s(p) p(1-p)) 
		\phi'(p) \\
		& \equiv \frac{1}{2} v(p) \phi''(p) + \mu(p) \phi'(p),
\end{split}     \end{equation}
where $\mu_{0}$ and $\mu_{1}$ are the scaled mutation rates from $A_{0}$ to $A_{1}$
and from $A_{1}$ to $A_{0}$, respectively, and $s(p)$ is the scaled and possibly 
frequency-dependent selection coefficient of $A_{1}$ relative to $A_{0}$.  In using the 
diffusion approximation, we assume that $N$ is large, that time is measured in units of 
$N$ generations, and the unscaled mutation rates and selection coefficient are of order 
$N^{-1}$.  Convergence results justifying the passage to the diffusion limit can be 
found in \cite{EK86}.

Population bottlenecks are transient events during which most of the population is descended 
from a small number of individuals.  On the diffusive time scale, these can be modeled 
as instantaneous jumps in the allele frequencies, and in this article we will be concerned 
with a class of models in which the bottlenecks always result in the temporary fixation of one 
of the two alleles, i.e., $p(\cdot)$ always jumps to $0$ or $1$.  We have two scenarios in mind.  
In the first, we consider a locus that is part of a non-recombining segment of DNA (e.g., 
a mammalian mitochondrial genome) subject to strong selective sweeps which occur at 
rate $\lambda$.  During each sweep, a unique copy of a favorable mutation arises at 
some linked site and rises rapidly to fixation.  Depending on whether the new, strongly-selected 
mutation occurs on a chromosome carrying an $A_{1}$ or $A_{2}$ allele, the frequency 
of $A_{1}$ will either increase from $p$ to $1$ with probability $p$ or decrease from
$p$ to $0$ with probability $1-p$.  Here we imagine that the selective advantage of 
the favored mutation is so strong that this change can be treated as a jump.  The 
pseudohitchhiking model introduced by Gillespie \cite{Gil00} belongs to this class, as 
does a related, more general model studied by Kim \cite{Kim04}.

The second scenario concerns demographic bottlenecks that occur during transmission 
of parasites from infected to uninfected hosts.  Here we will let $p$ denote the frequency 
of $A_{1}$ in a chronological series of infected hosts linked by a transmission chain, and 
we will assume that $p(\cdot)$ can be modeled by a diffusion process from the time when 
one of these hosts is first infected to the time when that host first transmits the infection to 
the next host in the transmission chain.  Suppose that transmissions occur at rate $\lambda$, 
and that each new infection is founded by a single parasite, as has been proposed for HIV-1 
\cite{Yuste99} and for some bacterial pathogens \cite{Rubin87}.  In this case, $p$ will jump to 
$0$ or $1$ following each transmission depending on the type of the transmitted parasite.  
Also, to allow for the possibility that transmission itself might be selective (e.g., \cite{Rong07}), 
we will let $w(p)$ denote the probability that the transmitted parasite is of type $A_{1}$ given 
that the frequency of this allele in the transmitting host is $p$.  In general, we stipulate that 
$w(0) = 0$, $w(1) = 1$, and that $w(p)$ is monotonically increasing.  If transmission is
unbiased, then $w(p) = p$, as in the pseudohitchhiking model.  A particular case of this 
transmission chain model was studied by Rouzine and Coffin \cite{RoCo99} to understand 
the effects of selection and transmission bottlenecks on antigenic variation in HIV-1.

Both of these scenarios can be modeled by a jump-diffusion process with infinitesimal
generator 
\begin{equation}  \begin{split}  \label{eq:G}
	G \phi(p)  = & \frac{1}{2} p(1-p) \phi''(p) + (\mu_{0} (1-p) - \mu_{1} p + s(p)p(1-p)) \phi'(p) + \\  
		 & \lambda w(p) \big( \phi(1) - \phi(p) \big) + \lambda (1 - w(p)) \big( \phi(0) 
		- \phi(p) \big),
\end{split}     \end{equation}
where for technical reasons we will assume that $s(p)$ and $w(p)$ are smooth functions 
on $[0,1]$, and that both mutation rates, $\mu_{0}$ and $\mu_{1}$, are positive.  Under 
these conditions, it can be shown (cf.\ Lemma 3.1) that the process $p(\cdot)$ has a unique 
stationary distribution, $\pi(p) dp$, which has a density on $[0,1]$.  To characterize the 
structured coalescent process corresponding to this model, we need to identify the stationary 
time reversal of the process $p(\cdot)$.  Formally, this can be done by solving the following 
adjoint problem for the operator $\tilde{G}$:
\begin{equation}  \label{eq:adjoint problem}
	\int_{0}^{1} \psi(p) G \phi(p) \pi(p) dp = \int_{0}^{1} \phi(p) \tilde{G} \psi(p) \pi(p) dp,
\end{equation}
where $\phi$ is in the domain of $G$.  If $\tilde{G}$ generates a Markov process 
$\tilde{p}(\cdot)$, then this process will have the same law as the stationary time reversal 
of $p(\cdot)$ \cite{Nel58}.  When $\lambda = 0$, $p(\cdot)$ is a diffusion process and 
a simple calculation using integration-by-parts shows that $\tilde{G} = G$, demonstrating 
that the law of the diffusion is invariant under time-reversal, as remarked above.  However, 
if $\lambda > 0$, then for the adjoint condition~\eqref{eq:adjoint problem} to be satisfied 
for all $\phi \in \mathcal{C}^{2} (\mathbb{R}) \cap \mathcal{C}[0,1]$, we must instead set
\begin{equation}  \label{eq:G_tilde}
	\tilde{G} \psi(p) = \frac{1}{2}p(1-p) \psi''(p) + \mut(p) \psi'(p),
\end{equation}
where
\begin{equation}
	\mut(p) =  \frac{1}{\pi(p)} \big( p(1-p) \pi'(p)  +  (1-2p - \mu(p)) \pi(p) \big)
\end{equation}
and $\psi \in \mathcal{C}^{2}(\mathbb{R}) \cap \mathcal{C}[0,1]$ satisfies
\begin{equation*}
		 \psi(1) = \int_{0}^{1} \psi(p)  \left( \frac{w(p) \pi(p)}{\kappa} \right) dp  
		 	\mbox{\;\;\; and \;\;\;}
		 \psi(0) = \int_{0}^{1} \psi(p) \left( \frac{(1-w(p)) \pi(p)}{1 - \kappa} \right) dp
\end{equation*}
with $\kappa = \int_{0}^{1} w(p) \pi(p) dp$.  Although it is not immediately clear that
the operator defined by~\eqref{eq:G_tilde} is the generator of a Markov process, this 
calculation does show that the process incorporating bottlenecks is not invariant under 
time reversal.

To gain some insight into the qualitative behavior of the time-reversed process, it is 
useful to consider two heuristic descriptions.  We begin by observing that the behavior 
of $\tilde{p}(\cdot)$ depends strongly on whether the boundary points $\{0, 1\}$ are 
accessible or inaccessible to the diffusive motion of the forward process.  Recall that 
for the Wright-Fisher diffusion corresponding to $A$ (which we call the diffusive motion 
of the jump-diffusion process), Feller's boundary classification conditions show that 
$0$ (resp.\ $1$) is accessible if and only if $u_{0} < 1/2$ (resp.\ $u_{1} < 1/2$),
see e.g.\ Section 4.7 in \cite{Ewe04}.  The importance of this distinction is illustrated 
in Figure 1, which shows sample paths of the jump-diffusion process corresponding 
to cases where the two boundaries are either inaccessible (A) or accessible (B) to the
forward diffusion.  To see what this suggests about the behavior of the time-reversed 
process, begin at the top of each figure and follow the sample path backwards in time 
towards the bottom.  If the boundaries are inaccessible, then whenever the sample path 
is followed back to a boundary at some time, the forward process will necessarily have 
reached that boundary via a jump.  Consequently, the time-reversed process must 
immediately jump into the interval $(0,1)$ whenever it arrives at a boundary that is 
inaccessible to the forward diffusion.  The behavior of the time-reversed process at a 
boundary that is accessible to the forward diffusion is very different.  In this case, when 
the sample path of the time-reversed jump diffusion hits that boundary, the forward 
process may have arrived there either diffusively or via a jump from the interior (Figure 
1B).  Accordingly, the time-reversed process need not immediately jump into the interior 
$(0,1)$ when it visits the boundary, although jumps can only occur when the process 
is on the boundary and are certain to occur at some such times if $\lambda > 0$.

\begin{figure}[hp]
	\begin{center}
	\includegraphics[scale=0.55]{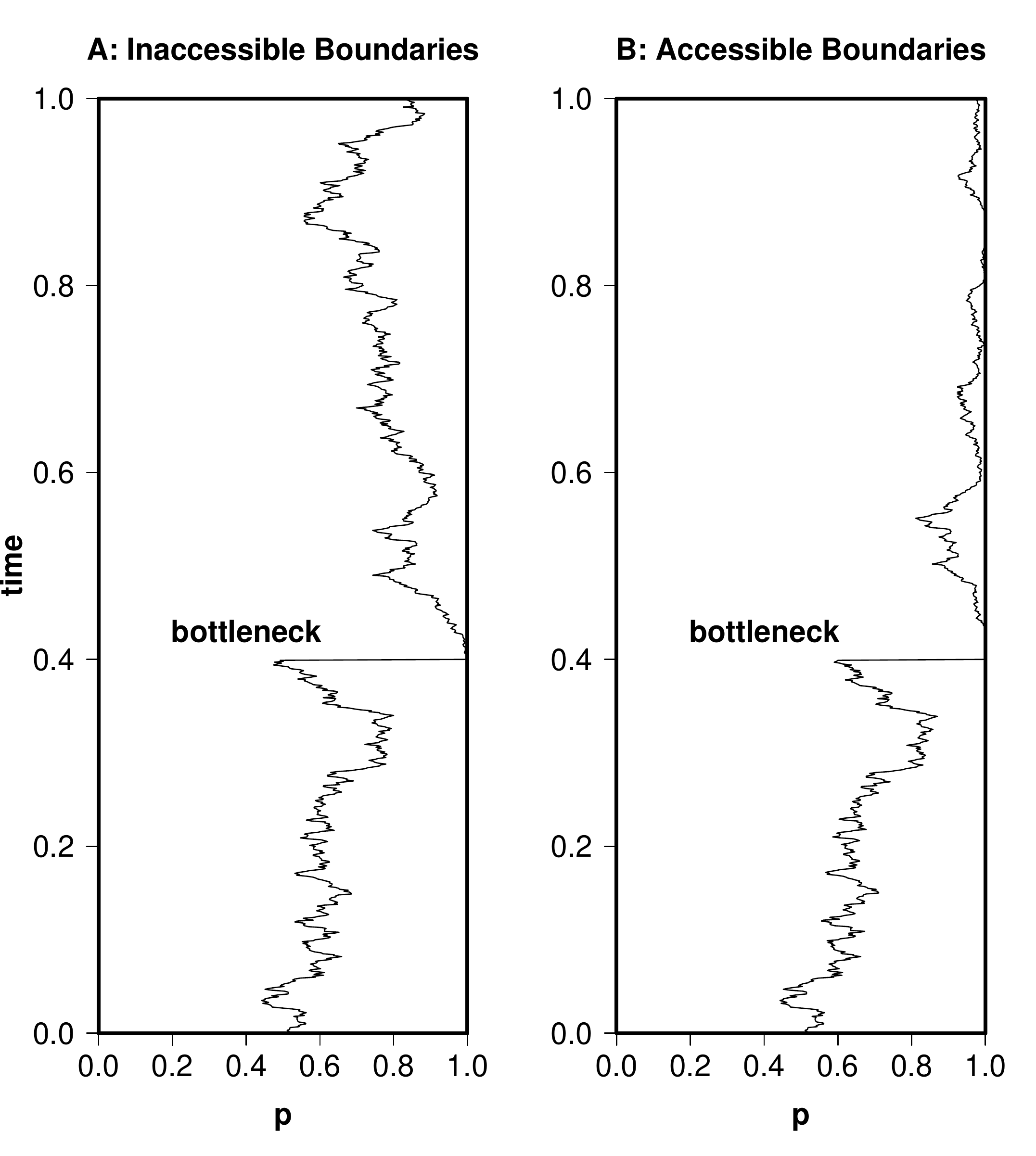}
	\caption{Sample paths of the jump-diffusion process~\eqref{eq:G} with either inaccessible 
	(A) or accessible boundaries (B). The forward diffusion is a neutral Wright-Fisher
	process with symmetric mutation: $\mu_{0} = \mu_{1} = 1$ in A and $0.2$ in B.} 
	\label{fig1}
	\end{center}
\end{figure}

A more quantitative picture of this second case can be obtained by considering a less 
singular process that approximates the jump-diffusion process corresponding to~\eqref{eq:G}.  
For $\epsilon \in (0, 1/2)$, let $p_{\epsilon}(\cdot) = \rob{p_{\epsilon}(t): t \geq 0}$ be
a perturbation of a Wright-Fisher diffusion which at rate $\lambda$ jumps to a point 
chosen uniformly at random from an interval of width $\epsilon$ adjacent to one of 
the two boundaries. More precisely, let $p_{\epsilon}(\cdot)$ be the Markov process 
with generator 
\begin{eqnarray*}
	G_{\epsilon} \phi(p) & = & \frac{1}{2} p(1-p) \phi''(p) + \big(\mu_{0}(1-p) - \mu_{1}p + 
		s(p) p(1-p) \big) \phi'(p) + \nonumber \\
		&  &  \lambda \left( w(p) \frac{1}{\epsilon} \int_{1-\epsilon}^{1} (\phi(q) - \phi(p)) dq  
		+  (1-w(p)) \frac{1}{\epsilon} \int_{0}^{\epsilon} (\phi(q) - \phi(p)) dq \right).
\end{eqnarray*}
Writing $\pi_{\epsilon}(p)$ for the density of the stationary distribution of this process, a 
simple calculation using~\eqref{eq:adjoint problem}
shows that the stationary time reversal of $p_{\epsilon}(\cdot)$, 
denoted $\tilde{p}_{\epsilon}(\cdot)$, is also a jump diffusion process with generator
\begin{eqnarray*}
	\tilde{G}_{\epsilon}\psi(p) & = & \frac{1}{2} p(1-p) \psi''(p)  + 
		\frac{1}{\pi_{\epsilon}(p)}	\big( p(1-p) \pi_{\epsilon}'(p) + 
		(1-2p-\mu(p)) \pi_{\epsilon}(p)\big) \psi'(p)  +   \nonumber \\
		&  &   \lambda \kappa_{\epsilon} \left( \frac{1}{\epsilon \pi_{\epsilon}(p)} 
		1_{(1-\epsilon,1]}(p) \right)  
		\int_{0}^{1} \left( \frac{w(q) \pi_{\epsilon}(q)}{\kappa_{\epsilon}} \right) 
		(\psi(q) - \psi(p)) dq  +  \nonumber \\
		& &   \lambda (1-\kappa_{\epsilon})  \left( \frac{1}{\epsilon \pi_{\epsilon}(p)} 
		1_{[0,\epsilon)}(p) \right)
		\int_{0}^{1}  \left( \frac{(1-w(q)) \pi_{\epsilon}(q)} {1-\kappa_{\epsilon}} \right) 
		(\psi(q) - \psi(p)) dq,
\end{eqnarray*}
where $\psi \in \mathcal{C}^{2}([0,1])$ and $\kappa_{\epsilon} = \int_{0}^{1} w(p) 
\pi_{\epsilon}(p)dp$.  It is easy to read off the behavior of this process from its generator.  
In particular, we see that $\tilde{p}_{\epsilon}(\cdot)$ can only jump when it is present in
the region $[0,\epsilon) \cup (1-\epsilon,1]$ and that the rate at which jumps occur out 
of this region is equal to $\lambda \kappa_{\epsilon}/(\epsilon \pi_{\epsilon}(p))$ when 
$p \in (1-\epsilon,1]$ and $\lambda (1-\kappa_{\epsilon})/(\epsilon \pi_{\epsilon}(p))$ 
when $p \in [0,\epsilon)$. 

To relate these observations to the process $\tilde{p}(\cdot)$, let $T > 0$ and notice that 
as $\epsilon$ tends to $0$, the sequence of processes $(p_{\epsilon}(\cdot))$ converges 
in distribution on $D_{[0,1]}([0,T])$ to $p(\cdot)$.  Furthermore, because time reversal is 
a continuous mapping on $D_{[0,1]}([0,T])$, the continuous mapping theorem \cite{EK86}
implies that the sequence of processes $(\tilde{p}_{\epsilon}(\cdot))$ converges in 
distribution to the process $\tilde{p}(\cdot)$.  In particular, this suggests that $\tilde{p}(\cdot)$ 
has the following behavior.  For each $\epsilon \in (0,1/2)$, define the additive functionals
\begin{eqnarray*}
	L_{1,\epsilon}(t) & \equiv &  \frac{1}{\epsilon} 
		\int_{0}^{t} \left(  \frac{1}{\pi(\tilde{p}(s))} \right) 
		1_{(1-\epsilon,1]}(\tilde{p}(s)) ds \nonumber \\
	L_{0,\epsilon}(t) & \equiv &  \frac{1}{\epsilon} 
		\int_{0}^{t} \left(  \frac{1}{\pi(\tilde{p}(s))} \right) 
		1_{[0,\epsilon)}(\tilde{p}(s)) ds,
\end{eqnarray*}
and suppose that for $i = 0, 1$, the limits $L_{i}(t) = \lim_{\epsilon \rightarrow 0} 
L_{i,\epsilon}(t)$ exist for all $t \geq 0$.  Here we would like to interpret $L_{i}(t)$ 
as the local time of the process $\tilde{p}(\cdot)$ at $i \in \{0,1\}$.  Then, by comparison 
with the jump-diffusion processes $\tilde{p}_{\epsilon}(\cdot)$, we expect that if both 
boundaries are accessible, then $\tilde{p}(\cdot)$ is a jump-diffusion with diffusive motion 
in $(0,1)$ governed by~\eqref{eq:G_tilde} which jumps from the boundary point $0$ 
to a random point in the interval $(0,1)$ distributed as $\tfrac{1}{\kappa}w(q)\pi(q)\,dq$ 
as soon as $L_0(\cdot)$ exceeds an exponential random variable with parameter 
$\ld\kappa$ and which jumps from the boundary point $1$ to a random point distributed 
as $\tfrac{1}{1-\kappa}\rob{1-w(q)}\pi(q)\,dq$ on $(0,1)$ as soon as $L_1(\cdot)$ 
exceeds an exponential random variable with parameter $\ld(1-\kappa)$.  Although 
these remarks are purely heuristic, we show below that they correctly describe the 
stationary time reversal of the pseudo-hitchhiking model and other jump-diffusions
with generators of the form (2).

%
%
\section{Main result}\label{sec:main_result}

Although our principle concern is with the modified Wright-Fisher process corresponding 
to~\eqref{eq:G}, we state our results for
a more general class of jump-diffusion processes, which 
we now introduce.  Let the forward process $\rob{p(t)\colon t\geq0}$ be the jump-diffusion 
process on $[0,1]$ corresponding to the generator
\begin{equation}  \begin{split}  \label{eq:G_general}
	G \phi(p)  = \tfrac12 v(p) \phi''(p) + \mu(p) \phi'(p)
   + \lambda w_0(p) \big( \phi(0) - \phi(p) \big) 
   + \lambda w_1(p) \big( \phi(1) - \phi(p) \big),
\end{split}     \end{equation}
for $\phi\in\C^2\rob{[0,1]}$.  In other words, the diffusive motion of $p(\cdot)$ is governed 
by the generator
\begin{equation}    \label{eq:A_general}
	A \phi(p) = \tfrac12 v(p) \phi''(p) + \mu(p) \phi'(p),\quad \phi\in\C^2\rob{[0,1]},
\end{equation}
with infinitesimal drift and variance coefficients, $\mu(\cdot)$ and $v(\cdot)$,
respectively, while
jumps occur at constant rate $\ld\geq0$ and move the process from state $p\in[0,1]$
either to $0$ with probability 
$w_0(p)\in[0,1]$ or to $1$ with probability $w_1(p):=1-w_0(p)$.
Throughout this article,
we will assume that the following conditions are satisfied.
\begin{assumpt}   \label{assumptions}
  The infinitesimal mean and variance satisfy $\mu(0)>0>\mu(1)$ 
  and $v(0)=v(1)=0<v(p)$ for all $p\in(0,1)$, respectively.
  Furthermore, $v(\cdot)$, $\mu(\cdot)$ and $w_0(\cdot)$ are analytic functions in a 
  neighborhood of $[0,1]$, and the infinitesimal variance has non-zero derivatives 
  $v^{'}(0)>0>v^{'}(1)$ at the boundaries.
\end{assumpt}
For example,
if $A$ is the generator of a neutral Wright-Fisher diffusion \eqref{eq:A} (with $s(p) 
\equiv 0$), then Assumption \ref{assumptions} is satisfied with $\mu(0) = \mu_0 > 0$, 
$\mu(1) = -\mu_1 < 0$, and $v^{'}(0) = 1 = -v^{'}(1)$.
We also remark that when Assumption~\ref{assumptions} is satisfied, Lemma 
\ref{l:equilibrium_distribution} shows that $\rob{p(t)\colon t\geq0}$ has a unique 
stationary distribution $\pi(p)dp$ with a density $\pi(\cdot)$ that satisfies a second 
order ordinary differential equation with non-local boundary conditions.

In Theorem 2.1, we characterize the time-reversed process $\rob{\pt(t)\colon t\geq0}$
of the forward process $\rob{p(t)\colon t\geq0}$.  In keeping with the heuristic
description given in the Introduction, $\rob{\pt(t)\colon t\geq0}$ is also a jump-diffusion 
process on $[0,1]$ but now with jumps from the boundary $\{0,1\}$ to the interior 
$(0,1)$.  The diffusive motion of this process is governed by the generator
\begin{equation}    \label{eq:A_tilde_general}
	\At \psi(p) = \tfrac{1}{2}v(p) \psi''(p) + \mut(p) \psi'(p)\qquad \text{where}\ \
  \mut(p):=-\mu(p)+\frac{(v\pi)^{'}(p)}{\pi(p)}
\end{equation}
and $\psi\in\C^2\rob{[0,1]}$.  Notice that this diffusion has the same infinitesimal 
variance as the forward diffusion, but has a different infinitesimal drift that depends
on the jump events via the stationary density $\pi(\cdot)$.  Also, the jump rates of the 
time-reversed process depend on a local time process which is described in the 
following way.  Recall that the scale function and the speed measure associated 
with $\At$ are
\begin{equation}  \label{eq:def:st_mt}
  \St(p):=\int_{\frac{1}{2}}^p
              \exp\roB{ -\int_{\frac12}^x\frac{2\mut(z)}{v(z)}\,dz }
           dx
  \ \text{ and }\ \mt(dp):=\frac{1}{v(p)\St^{'}(p)}dp,\quad p\in[0,1],
\end  {equation}
respectively.
The scale function will be identified with the associated measure $\St(dp):=\St^{'}(p)dp$ 
on $[0,1]$ and the speed measure $\mt(dp)$ will be identified with its density function.

We define the local time process of the jump-diffusion $\pt(\cdot)$ such that it agrees 
with the local time process of the diffusive motion until the first jump.  More formally,
we will introduce a non-negative process $\rob{\Lt_p(t)\colon t\geq0,\,p\in[0,1]}$ 
which is almost surely continuous in $(t,p)$ and which satisfies
\begin{equation}  \label{eq:occupation_time_formula_Lt}
  \int_0^t f\rob{\pt(u)}\,du=\int_0^1 f(p)\Lt_p(t)\,\mt(dp)\qquad\text{a.s.}\ t\geq0
\end  {equation}
for all measurable $f\colon[0,1]\to[0,\infty)$. 
We remark that  the local time process satisfying~\eqref{eq:occupation_time_formula_Lt}
differs from the semi-martingale local time of the diffusive motion
of the time-reversed process by a 
scalar factor (see Eq.~\eqref{eq:Lt_Lb}), i.e., $\Lt$ is a weighted semi-martingale local time.
That this process is well-defined is shown below 
in Lemma~\ref{l:local_time_process}.  The last ingredient needed in our construction 
is a pair of independent, exponentially-distributed random variables, $R_0$ and $R_1$, 
with parameters
\begin{equation}  \label{eq:rate}
  r_i:=\lim_{p\to i}\roB{\frac{\mt(p)}{\pi(p)}\lambda \kappa_i}\in[0,\infty]
\end{equation}
where $\kappa_i:=\int_0^1 w_i(p)\pi(p)dp$, $i\in\{0,1\}$.  The existence of the limit 
displayed in~\eqref{eq:rate} is guaranteed by Lemma~\ref{l:rate}.  By convention, 
$R_i:=0$ if $r_i=\infty$ and $R_i:=\infty$ if $r_i=0$.

With these definitions, we now describe the dynamics of the time-reversed process 
$\rob{\pt(t)\colon t\geq0}$.  Between jump times, $\rob{\pt(t)\colon t\geq0}$ evolves 
according to the law of the diffusion governed by $\At$.  If this diffusion hits a boundary $i\in\{0,1\}$ 
at a time $t\geq0$ and if at that time the local time process exceeds the random variable 
$R_i$, that is, if $\Lt_i(t)\geq R_i$, then $\pt(\cdot)$ jumps from $i$ to a random point
chosen from $(0,1)$ according to the distribution $\tfrac{1}{\kappa_i}\int w_i(p)\pi(p)dp$. 
From this point, $\pt(\cdot)$ restarts independently of the sample path up to that time.

To better understand how the dynamics of $\pt(\cdot)$ are influenced by the boundary 
behavior of the forward process, we take a closer look at the jump times.  Because 
the coefficients $v(\cdot)$ and $\mu(\cdot)$ are smooth on an interval containing $[0,1]$, an 
application of Feller's boundary classification criteria shows that a boundary point 
$i\in\{0,1\}$ is accessible to the forward diffusive motion if and only if $2\abs{\mu(i)}
<\abs{v^{'}(i)}$.  Then, in conjunction with Lemma~\ref{l:boundary_behavior_pi}, 
which describes the asymptotics of the density $\pi(p)$ near the boundaries, 
Lemma~\ref{l:rate} implies that
\begin{equation}   \label{eq:value_of_rate}
    r_i:=\lim_{p\to i}\roB{\frac{\mt(p)}{\pi(p)}\lambda \kappa_i}
    \left\{   \begin{array}{ll}
  		  \in(0,\infty)
                   & \mbox{ \; if $2\abs{\mu(i)} < \abs{v^{'}(i)}$
                                 and $\lambda w_i(\cdot)\not\equiv0$ }  \\
        =\infty    & \mbox{ \; if $2\abs{\mu(i)} \geq \abs{v^{'}(i)}$
                              and $\lambda w_i(\cdot)\not\equiv0$}\\
        =0         & \mbox{ \; if $\lambda w_i(\cdot)\equiv0$}
	\end{array}  \right.
\end{equation}
for $i\in\{0,1\}$.
Thus, provided that $\lambda w_i(\cdot)\not\equiv0$, the time-reversed process immediately 
jumps into the interior $(0,1)$ if the boundary point  is inaccessible to the forward diffusive 
motion, that is, if $2\abs{\mu(i)}\geq\abs{v^{'}(i)}$.  In this case, the state space of $\pt(\cdot)$ 
is in fact $[0,1]\setminus\{i\}$.  In contrast, if $i$ is accessible to the forward diffusion and 
$\lambda w_i(\cdot) > 0$, then the exponential random variable $R_{i}$ is almost surely 
positive and so a positive amount of local time will have to be accrued at $i$ before a
jump occurs off of this boundary point.

Notice that, in either case, we expect that both boundary points are accessible to the
backward diffusive motion.  According to Lemma~\ref{l:boundary_behavior_mue_tilde}
\begin{equation}   \label{eq:mue_tilde_zero}
    \mut(i)=
    \left\{   \begin{array}{ll}
		  \mu(i) & \mbox{ \; if $2\abs{\mu(i)} \leq \abs{v^{'}(i)}$} \\
		  v^{'}(i)-\mu(i) & \mbox{ \; if $2\abs{\mu(i)} \geq \abs{v^{'}(i)}$} 
	\end{array}  \right.,\quad i\in\{0,1\},
\end{equation}
and again an application of Feller's boundary criteria shows that the boundary point $i$ 
is accessible to the backward diffusive motion whenever $2\mu(i)\neq v^{'}(i)$.
The critical case is more subtle.  Then, $2\mut(i)=v^{'}(i)$, and so $i$ would be 
inaccessible if the drift coefficient $\mut(\cdot)$ were analytic in a neighborhood of $i$.
However, we show in Lemma~\ref{l:boundary_behavior_mue_tilde} that
\begin{equation}
  \mut(p)= \mu(i)+ \frac{v^{'}(i)}{\ln\rob{|p-i|}}
           +O\rob{\frac{\abs{p-i}}{\ln{\abs{p-i}}}},
\end{equation}
and then Feller's criteria reveal that the logarithmic singularity is just sufficient to render
the point $i$ accessible to the backward diffusive motion when $2\mut(i)=v^{'}(i)$.

Our main result states that the process $\pt(\cdot)$ has the same law as the stationary 
time reversal of the jump-diffusion $p(\cdot)$.
\begin{thm}  \label{thm:main_result}
  Assume~\ref{assumptions}.
  Let $p(\cdot)$ be the jump-diffusion on $[0,1]$ with generator $G$ as defined
  in~\eqref{eq:G_general}.
  Then the process $\rob{\pt(t)\colon t\geq0}$ is a version of the
  stationary time reversal of $\rob{p(t)\colon t\geq0}$,
  that is,
  \begin{equation}
    \rob{\pt(t)\colon t\leq T}\eqd\rob{p(T-t)\colon t\leq T}\quad \fa T\geq0
  \end{equation}
  if the distribution of $p(0)$ is the stationary distribution $\pi(p)dp$.
\end{thm}
The proof of Theorem~\ref{thm:main_result} is 
deferred to Section~\ref{sec:the_backward_process}.

Theorem~\ref{thm:main_result} establishes the time reversal of the stationary
process over a fixed time interval $[0,T]$, $T<\infty$ fixed and non-random.
Readers being interested in other pathwise time reversals are referred
to the literature.
It has been shown that processes which are in 'Hunt duality' (see \cite[Chapter VI]{BG68})
are time reversals of each other.
Reversing time at the end point of an excursion from an accessible boundary point results
in the dual process being started at this boundary point, see \cite{GS81,Mi81}.
The paper of Mitro~\cite{Mi84} reverses time at inverse local time points.

The remainder of the paper is organized as follows.  The next section collects
some results concerning the stationary distribution of the jump-diffusion process~\eqref{eq:G}.
Section~\ref{sec:Jump times at inaccessible boundaries} describes the boundary 
behavior of $\pt(\cdot)$.  In particular we show that the time-reversed process jumps
immediately off of any boundary that is inaccessible to the forward diffusion.  In 
Section~\ref{sec:the generator of the time-reversed process} we identify a core for 
the generator $\tilde{G}$ satisfying the adjoint condition~\eqref{eq:adjoint problem}.
The local time process 
of $\pt(\cdot)$ is introduced and studied in Section~\ref{sec:The local time process}.
Finally, Section~\ref{sec:the_backward_process} shows that $\pt(\cdot)$ has generator 
$\Gt$.  The proof of this result depends on an application of the It\^o-Tanaka formula.

%
%
\section{The stationary distribution}

The following lemma asserts that, if the conditions of Assumption~\ref{assumptions} are satisfied,
then the jump-diffusion process $p(\cdot)$ has a unique stationary distribution on $[0,1]$.
It is also shown that this distribution has a density $\pi(\cdot)$ with respect to Lebesgue 
measure which satisfies a second-order ordinary differential equation (ODE) subject
to boundary conditions that are non-local whenever $\ld > 0$.  If $\ld = 0$, then this 
equation can be solved explicitly, leading to the familiar expression
\begin{equation}
	\pi(p) = C^{-1} \frac{1}{v(p)} \exp \Big( 2 \int^{p} \mu(q)/v(q) dq \Big),
\end{equation}
where $C$ is a normalizing constant, e.g., see Section 4.5 in \cite{Ewe04}.  Although 
a general closed-form expression for $\pi(\cdot)$ apparently does not exist when 
$\ld > 0$, $\pi(\cdot)$ can be calculated by numerically solving~\eqref{eq:pi}
using a modification of the shooting method \cite{Press92}.  In addition, below we 
give an explicit formula for the stationary density in the important special case of 
a neutral Wright-Fisher diffusion subject to recurrent bottlenecks.

%
%
\begin{lem}  \label{l:equilibrium_distribution}
  Assume~\ref{assumptions}.
  Then there exists a unique stationary distribution for the pro\-cess
  $\rob{p(t)\colon t\geq0}$. This distribution is given by $\1_{(0,1)}(p)\pi(p)dp$
  where $\pi\colon(0,1)\to(0,\infty)$ is the unique solution
  of the non-local boundary value problem
  \begin{equation}  \begin{split}  \label{eq:pi}
    \rob{(\tfrac12 v\pi)^{''}-(\mu\pi)^{'}-\lambda\pi}(p)&=0
           \quad \fa p\in(0,1)\\
    \lim_{p\to0}\rob{\mu\pi-(\tfrac12 v\pi)^{'}}(p)&=\ld\kappa_0\\
    \lim_{p\to1}\rob{\mu\pi-(\tfrac12 v\pi)^{'}}(p)&=-\ld\kappa_1\\
    \lim_{p\to0}(v\pi)(p)&=0=\lim_{p\to1}(v\pi)(p)\\
    \int_0^1 \pi(p)dp&=1,
  \end  {split}     \end  {equation}
  where $\kappa_i:=\int_0^1 w_i(p)\pi(p)dp$ for $i\in\{0,1\}$.
  Furthermore $p(t)$ converges in distribution to the stationary
  distribution as $t\to\infty$ for every initial distribution of $p(0)$.
\end  {lem}
\begin{proof}
  Existence and uniqueness of a stationary distribution
  $\pib(dp)$ follow from standard arguments, so we only give a sketch.
  Couple two versions of $\rob{p(t)\colon t\geq0}$ with different initial distributions
  through the same jump times such that the diffusive motions in between
  jumps are independent until they first meet and are identical thereafter.
  Due to the assumption $\mu(0)>0>\mu(1)$, the coupling is successful
  if there are no jumps, that is, if $\lambda=0$, see Theorem V.54.5 in~\cite{RW2}.
  In the presence of jumps $(\lambda>0)$, the probability that both components 
  jump to the same boundary is positive at every jump and, therefore, the two 
  components agree eventually.  As a consequence of this successful coupling 
  and of compactness of $[0,1]$, $p(t)$ converges in distribution to a probability 
  measure $\pib(dp)$ as $t\to\infty$ and $\pib(dp)$ is an invariant distribution.

  Next we prove that $\pib(\cdot)$ has a smooth density $\pi(\cdot)$.
  Denote by $(X(t))_{t\geq0}$ the diffusion governed by $A$ (see \eqref{eq:A}).
  The scale function and the speed measure associated with $A$ are
  \begin{equation}  \label{eq:def:s_m}
    S(p):=\int_{\frac{1}{2}}^p
                \exp\roB{ -\int_{\frac12}^x\frac{2\mu(z)}{v(z)}\,dz }
             dx
    \ \text{ and }\ m(p)dp:=\frac{1}{v(p)S^{'}(p)}dp,\quad p\in[0,1],
  \end  {equation}
  respectively. Existence and smoothness of the density $\pi(\cdot)$ will be
  derived from existence and uniqueness of the transition density $Q(t;p,q)$
  of $(X(t))_{t\geq0}$ with respect to the speed measure.
  Existence of $Q(t;p,q)$ is established in It\^o and McKean (1974) \cite{ItoMcKean}
  (\cite{McK56} is more detailed in a special case)
  via an eigen-differential expansion.
  To state this result more formally, we introduce the following notation.
  The interval defined in \cite{ItoMcKean} -- here denoted by $I^\bullet$ --
  is the unit interval
  closed at $0$ if $0$ is accessible, closed at $1$ if $1$ is accessible and open
  otherwise.
  For this, note that whenever $(X(t))_{t\geq0}$ hits a boundary point, it
  immediately returns to the interior $(0,1)$ because of the assumption $\mu(0)>0>\mu(1)$.
  Moreover note that the stopping time $\min\{t\geq0\colon X(t)\not\in I^\bullet\}=\infty$
  is infinity almost surely.
  The generator of $(X(t))_{t\geq0}$ is defined in \cite{ItoMcKean} via right derivatives.
  As $(X(t))_{t\geq0}$ is a regular diffusion, this generator coincides with
  \begin{equation}
    A f(p)=\frac{1}{m(p)}\frac{d}{dp}\roB{\frac{1}{S^{'}(p)}f^{'}(p)}\qquad p\in I^\bullet
  \end{equation}
  for $f\in\C^2(I^\bullet)$.
  There exists a solution
  $\mathfrak{e}(\gamma,\cdot)=\rob{\mathfrak{e}_1(\gamma,\cdot),\mathfrak{e}_2(\gamma,\cdot)}$
  of
  \begin{equation}  \begin{split}
    \roB{A\mathfrak{e}\rob{\gamma,\cdot}}(p)&=\gamma\mathfrak{e}\rob{\gamma,p}
    \qquad\fa 0<p<1\\
    \mathfrak{e}\rob{\gamma,\tfrac{1}{2}}=(1,0)&\qquad
    \frac{1}{m(\tfrac{1}{2})}\mathfrak{e}^{'}\rob{\gamma,\tfrac{1}{2}}=(0,1)
  \end{split}     \end{equation}
  for every $\gamma\in(-\infty,0]$ such that $\gamma\mapsto\mathfrak{e}(\gamma,p)$
  is continuous for every $p\in I^\bullet$.
  Based on these eigenfunctions,
  it is shown in \cite{ItoMcKean} that there exists a Borel measure
  $\mathfrak{s}(d\gamma)$ from $(-\infty,0]$ to $2\times2$ symmetric non-negative definite
  matrices
  \begin{equation}
    \mathfrak{s}(d\gamma)=\begin{pmatrix}
                            \mathfrak{s}_{11}(d\gamma) & \mathfrak{s}_{12}(d\gamma) \\
                            \mathfrak{s}_{21}(d\gamma) & \mathfrak{s}_{22}(d\gamma)
                          \end{pmatrix}
  \end{equation}
  such that
  \begin{equation}
    Q(t;p,q)=\int_{-\infty}^0 e^{\gamma t} \mathfrak{e}^T(\gamma,p)\mal\mathfrak{s}(d\gamma)
                                          \mal\mathfrak{e}(\gamma,p),
                            \quad(t,p,q)\in(0,\infty)\times I^\bullet\times I^\bullet,
  \end{equation}
  is the transition density of $(X(t))_{t\geq0}$ with respect to the speed measure $m(\cdot)$.
  Now as our jump diffusion $p(\cdot)$ could also jump to an inaccessible boundary, we need
  to extend $p\mapsto Q(t;p,q)$ onto $[0,1]$. 
  Note that if $i\in\{0,1\}$ is inaccessible,
  then $i$ is an entrance boundary due to the assumption $(-1)^i\mu(i)>0$.
  As in Problem 3.6.3 in \cite{ItoMcKean}, one uses the Markov property to extend
  $(X(t))_{t\geq0}$ to the state space $[0,1]$.
  Thus we may assume $Q(t;p,q)$ to be defined on $(0,\infty)\times[0,1]\times I^\bullet$.

  With these results on the transition density of $(X(t))_{t\geq0}$, we now
  establish existence of a smooth density of $\pib(\cdot)$.
  Define $\kappa_0\in[0,1]$, $\kappa_1:=1-\kappa_0$ by
  \begin{equation}
    \kappa_i:=\int_0^1 w_i(p)\pib(dp)\qquad\text{for }i\in\{0,1\}
  \end{equation}
  and observe that $\kappa_i$ is the probability that a stationary version
  of the process jumps to the boundary point $i$ when it jumps.
  Recall that the jump times of $p(\cdot)$ form a Poisson process with rate
  $\lambda$ and that in between jumps, $p(\cdot)$ evolves according to $A$.
  If $U$ is any Borel measurable set in $[0,1]$, then by conditioning
  on the time and distribution of the last jump, we have
  \begin{equation*}
    \pib(U)=\kappa_0\int_0^\infty \lambda e^{-\lambda t}
            \int_U Q(t;0,q)m(q)dq\,dt
            +\kappa_1\int_0^\infty \lambda e^{-\lambda t}
            \int_U Q(t;1,q)m(q)dq\,dt.
  \end{equation*}
  Interchanging integrals, we infer that
  $\pib(\cdot)$ has a density with respect to Lebesgue measure
  and we set $\pi(q)dq:=\pib(dq)$ where $\pi\colon(0,1)\to[0,\infty)$
  satisfies
  \begin{equation}  \begin{split}
    \pi(q)&=\sum_{j=0}^1\kappa_j m(q)\int_0^\infty \ld e^{-\ld t}\int_{-\infty}^0 e^{\gamma t}
            \mathfrak{e}^T(\gamma,j)\mathfrak{s}(d\gamma)\mathfrak{e}(\gamma,q)\,dt\\
          &=\sum_{j=0}^1 \kappa_j m(q)\int_{-\infty}^0 \frac{\lambda}{\lambda-\gamma}
            \mathfrak{e}^T(\gamma,j)\mathfrak{s}(d\gamma)\mathfrak{e}(\gamma,q)\,dt\\
          &=\lambda m(q)\roB{\kappa_0 G_\lambda(0,q)+\kappa_1 G_\lambda(1,q)}.
  \end{split}     \end{equation}
  The function $G_\lambda(p,q)$ is the Green's function and is $\C^2$ in
  the second variable for every $p\in[0,1]$.
  As the speed density $m(\cdot)$ is also $\C^2$ in $(0,1)$
  due to Assumption~\ref{assumptions}, we conclude that the
  stationary density $\pi(\cdot)$ is twice continuously differentiable.

  The main step of the proof is to show that $\pi(\cdot)$ satisfies~\eqref{eq:pi}.
  By Proposition 4.9.2 of Ethier and Kurtz~\cite{EK86},
  the stationary distribution $\pi(p)dp$ satisfies
  \begin{equation} \label{eq:A_phi_pi}
    \int_0^1 G\phi(p)\pi(p)dp=0
  \end  {equation}
  for all $\phi\in\mathbf{C}^2\rob{[0,1]}$.
  Let $0<\eps<\tfrac12$.
  The functions $v,\mu,\phi$ and $\pi$ are $\mathbf{C}^2$
  in $[\eps,1-\eps]$.
  Integration by parts yields
  \begin{equation}  \begin{split} \label{eq:int_by_parts}
    \lefteqn{\int_\eps^{1-\eps}G\phi\mal\pi \,dp
       -\phi(0)\int_\eps^{1-\eps}\lambda w_0\mal\pi \,dp
       -\phi(1)\int_\eps^{1-\eps}\lambda w_1\mal\pi \,dp}\\
     &=\int_\eps^{1-\eps}\phi^{''}\mal (\tfrac12 v\pi) +\phi^{'}\mal(\mu\pi)
       -\lambda\phi\pi \,dp\\
     &=[\phi^{'}\tfrac12 v\pi]_\eps^{1-\eps}
       + [\phi\mal\rob{\mu\pi-(\tfrac12 v\pi)^{'}}]_\eps^{1-\eps}
       -\int_\eps^{1-\eps}\phi\mal\eckB{(\mu\pi)^{'}-(\tfrac12 v\pi)^{''}
       +\lambda\pi}\,dp.
  \end  {split}     \end  {equation}
  By considering all functions $\phi \in \mathbf{C}^2$ with support in $(\eps,1-\eps)$
  and then letting $\eps\to0$, we conclude that $\pi(\cdot)$ satisfies the second-order 
  ODE in~\eqref{eq:pi}.  Furthermore, because the functions $G \phi$, $w_0$, $w_1$ 
  are bounded and $\pi$ is integrable, we may apply the dominated convergence theorem 
  to theintegrals on the left-hand side of~\eqref{eq:int_by_parts} as $\eps\to0$.
  Together with~\eqref{eq:A_phi_pi} this shows that
  \begin{equation}  \begin{split}
    \limepsO &[\phi^{'}\tfrac12 v\pi]_\eps^{1-\eps}
       + \phi(1)\mal\limepsO\rob{\mu\pi-(\tfrac12 v\pi)^{'}}(1-\eps)
       - \phi(0)\mal\limepsO\rob{\mu\pi-(\tfrac12 v\pi)^{'}}(\eps)\\
    &=-\phi(1)\lambda\kappa_1-\phi(0)\lambda\kappa_0.
  \end{split}     \end{equation}
  As $\phi$ was arbitrary this implies the non-local boundary conditions
  in~\eqref{eq:pi}.

  If $\pih(\cdot)$ is another normalized solution of~\eqref{eq:pi}, then
  reversing the previous arguments shows that~\eqref{eq:A_phi_pi}
  holds with $\pi$ replaced by $\pih(\cdot)$. This in turn
  implies that $\pih(p)dp$ is another stationary distribution 
  and we conclude that $\pih=\pi$. It remains to show that $\pi(\cdot)$
  is strictly positive. Assuming $\pi(p)=0$ for some $p\in(0,1)$,
  we conclude that $\pi^{'}(p)=0$ from $p$ being necessarily a global
  minimum. However, the only solution of the second-order ODE
  in~\eqref{eq:pi} satisfying $\pi(p)=0=\pi^{'}(p)$ is the zero function,
  which contradicts the assumption that $\pi(\cdot)$ is a probability density.
  \qed
\end  {proof}

\begin{rem}
Lemma 3.1 can be used to find an explicit formula for $\pi(\cdot)$ when the jump-diffusion
process is a model of a neutrally-evolving population subject to recurrent bottlenecks, i.e., 
when $p(\cdot)$ has generator
\begin{equation*}
	G \phi(p) = \frac{1}{2} p(1-p) \phi''(p) + (\mu_{0} (1-p) - \mu_{1} p) \phi'(p) +
		\lambda \Big( p \phi(1) + (1-p) \phi(0) - \phi(p) \Big).
\end{equation*}
In this case, \eqref{eq:pi} is a hypergeometric equation and, using the fact that the mean frequency
of allele $A_{1}$ in a stationary population is $\mu_{0}/(\mu_{0} + \mu_{1})$, we find that the
density $\pi(p)$ is equal to 
\begin{equation*}
	\pi(p) = C^{-1} p^{2\mu_{0}-1}(1-p)^{2\mu_{1}-1} \Big[ \mu_{0} F(1-a,1-b,2\mu_{0},p) +
		\mu_{1} F(1-a,1-b,2 \mu_{1}, 1-p) \Big],
\end{equation*}
where $C$ is a normalizing constant, $F(a,b,c;z)$ is Gauss' hypergeometric function, and
the constants $a$ and $b$ are determined (up to interchange) by the equations $a+b = 
3 - 2(\mu_{0} + \mu_{1})$ and $ab = 2(\lambda+1-\mu_{0}-\mu_{1})$.
\end{rem}

The second lemma of this section provides information on the boundary behavior
of the density of the stationary distribution. This information is derived using results on 
second-order ODEs with regular singular points.

We adopt the Landau big-O and little-o notation. In addition, for two functions $\psi_1(\cdot)$ 
and $\psi_2(\cdot)$, we write $\psi_1(p)\sim\psi_2(p)$ as $p\to i$ if both $\psi_1(p)=O\rob{\psi_2(p)}$
and $\psi_2(p)=O\rob{\psi_1(p)}$ as $p\to i$.
%
%
\begin{lem}   \label{l:boundary_behavior_pi}
  Assume~\ref{assumptions}.
  Let $\pi(\cdot)$ be the density of the stationary distribution of the 
  jump-diffusion $p(\cdot)$ corresponding to the generator~\eqref{eq:G_general}.
  Then, for $i\in\{0,1\}$, $\pi(\cdot)$ is equal to
  \begin{equation}   \label{eq:boundary_behavior_pi}
    \pi(p)  =  \left\{   \begin{array}{ll}
       C_i\abs{p-i}^{\frac{2\mu(i)-v^{'}(i)}{v^{'}(i)}}
  		     +O(1)
                  & \mbox{if $2\abs{\mu(i)} < \abs{v^{'}(i)}$}\\
  		 \frac{2\lambda\kappa_i}{\abs{v^{'}(i)}}\ln\rob{\frac{1}{|p-i|}}+O(1)
                  & \mbox{if \ \,$2\mu(i) = v^{'}(i)$ }\\
  		 \frac{2\lambda\kappa_i}{\abs{2\mu(i)-v^{'}(i)}}
           +O\rob{\abs{p-i}+
                 \abs{p-i}^{\frac{2\mu(i)-v^{'}(i)}{v^{'}(i)}}
                 }
                  & \mbox{if $2\abs{\mu(i)} > \abs{v^{'}(i)}$
                          and $\lambda w_i\not\equiv0$}\\
        C_i\abs{p-i}^{\frac{2\mu(i)-v^{'}(i)}{v^{'}(i)}}
           +O\rob{\abs{p-i}^{\frac{2\mu(i)}{v^{'}(i)}}}
                  & \mbox{if $2\abs{\mu(i)} > \abs{v^{'}(i)}$
                          and $\lambda w_i\equiv0$}
  	\end{array}  \right.
  \end{equation}
  as $p\to i$ where $C_i\in(0,\infty)$.
  In addition if $2\mu(i)=v^{'}(i)$ and $\lambda w_i\equiv 0$,
  then $\pi(i)>0$.
\end{lem}
\begin{proof}
We only consider $i=0$ as the case $i=1$ is analogous.
We begin by observing that $i = 0$ is a regular singular point 
for the differential equation in~\eqref{eq:pi}, see e.g.\ Section 9.6
in~\cite{BiRo89} for this concept.
The associated indicial equation for $0$ is
\begin{equation}
  \nu(\nu-1)+2\frac{v^{'}(0)-\mu(0)}{v^{'}(0)}\mal\nu=0
\end{equation}
and has roots
$\alpha:=0$ and $\beta:=\tfrac{2\mu(0) - v^{'}(0)}{v^{'}(0)}$.
Note that $\beta>-1$.
If $\al-\beta\not\in\Z$, then Theorem IX.7 in~\cite{BiRo89} tells us
that $\pi(\cdot)$ is equal to a linear combination of
$b_1(p):=p^\al\rob{1+h_1(p)}$ and $b_2(p):=p^\beta\rob{1+h_2(p)}$ in a neighborhood
of $0$ where $h_1$ and $h_2$ are suitable analytic functions
satisfying $h_1(0)=0=h_2(0)$.

If $\al-\beta\in\Z$, then Theorem IX.8 in~\cite{BiRo89} shows
that $\pi(\cdot)$ is equal to a linear combination of
$b_1(p)$, $b_2(p)$ and $\ln(p) b_1(p)$ in a neighborhood
of $0$.
If $2\mu(0)/v^{'}(0)\in\Nat_{\geq2}$, then assuming
$\pi(p)=-c_1\ln(p)+O(1)$, $\pi^{'}(p)=-c_1\tfrac{1}{p}+O(1)$
for some constant $c_1>0$ leads
to the contradiction
\begin{equation}
  \infty>\ld\kappa_0=\rob{\mu(0)-\tfrac12 v^{'}(0)}\lim_{p\to0}\pi(p)
                     -\tfrac12 v^{'}(0)\lim_{p\to0}p\pi^{'}(p)=\infty
\end{equation}
where we have used~\eqref{eq:pi}.
Therefore $\ln(p)$ does not contribute to $\pi(\cdot)$ if
$2\mu(0)>v^{'}(0)$.

It remains to calculate the coefficients.
In the case $2\mu(0) \neq v^{'}(0)$, insert
$\pi(b)=c_1 b_1(p)+c_2 b_2(p)$ into~\eqref{eq:pi} to obtain
the coefficient $c_1$
\begin{equation}  \begin{split}
  \lambda\kappa_0
  &=\lim_{p\to0}
     \eckB{ \rob{\mu(p)-\frac{1}{2}v^{'}(p)}\pi(p)-\frac{1}{2}v(p)\pi^{'}(p) }\\
  &=\lim_{p\to0}
     \eckB{ \rob{\mu(0)-\frac{1}{2}v^{'}(0)}\rob{c_1+c_2p^\beta}
            -\frac{1}{2}v^{'}(0)p c_2\beta p^{\beta-1} }\\
  &=\rob{\mu(0)-\frac{1}{2}v^{'}(0)}c_1.
\end{split}     \end{equation}
Of course if $\lambda\kappa_0=0$, then $\pi(\cdot)\not\equiv0$
implies $c_2>0$.
Next we show that $\lambda\kappa_0>0$ together with
$2\mu(0)< v^{'}(0)$ implies $c_2>0$.

Assuming $c_2=0$ implies that
$\pi(0)=\tfrac{2\lambda\kappa_0}{2\mu(0)-v^{'}(0)}<0$
which contradicts $\pi(\cdot)$ being a density function.
In the critical case $2\mu(0)=v^{'}(0)$, \eqref{eq:pi} implies that
\begin{equation}
  \lambda\kappa_0=-\frac{1}{2}v^{'}(0)\lim_{p\to0}\rob{p\pi^{'}(p)}.
\end{equation}
Therefore the coefficient of $-\ln(p)$ is $\tfrac{2\lambda\kappa_0}{v^{'}(0)}$.
If $2\mu(0)=v^{'}(0)$ and $\lambda\kappa_0=0$,
then assuming $\pi(0)=0$  implies $\pi(p)=c p^n+O(p^{n+1})$ with $c\neq 0$ and
$n\geq1$.
Inserting this into the ODE in~\eqref{eq:pi} leads to
\begin{equation}  \begin{split}
  0&=\frac{1}{2}v^{'}(0)c \rob{p^{n+1}}^{''}-\mu(0)c\rob{p^n}^{'}+O\rob{p^n}\\
  &=\frac{(n+1)n}{2}v^{'}(0) c p^{n-1}-n\frac{v^{'}(0)}{2}p^{n-1}+O\rob{p^n}
\end{split}     \end{equation}
as $p\to0$.
Dividing by $nv^{'}(0) c p^{n-1}/2$ and letting $p\to0$ results in
the contradiction $n+1=1$.
\qed
\end{proof}

%
%
\noindent
\section{Boundary behavior of the time-reversed process}%
\label{sec:Jump times at inaccessible boundaries}

We begin this section by characterizing the boundary behavior of the infinitesimal drift 
coefficient of the time-reversed process.  This information is of interest for two reasons.  First,
it will be used to establish that any boundary point that is accessible to the forwards-in-time
process, either diffusively or via jumps, is accessible to the diffusive motion of the 
time-reversed process.  Secondly, we also expect the time-reversed process to have 
the same state space, $[0,1]$, as the forward process.  Indeed, if a boundary point $i$ 
is inaccessible to the forward diffusive motion, then subsequent results will show that 
the time-reversed process jumps back into the interior as soon as it hits a boundary.  
If, however, $i$ is accessible to the forward diffusive motion, then because the time-reversed 
process may visit $i$ without jumping, we need to confirm that $\tilde{p}(\cdot)$ does 
not then wander outside of $[0,1]$.  To this end, we will show that $\mut(0) \geq 0$ 
whenever $0$ is accessible and similarly that $\mut(1) \leq 0$ whenever $1$ is
accessible.

%
%
\begin{lem}   \label{l:boundary_behavior_mue_tilde}
  Assume~\ref{assumptions}.
  Then the drift function $\mut(\cdot)$ of the backward diffusive motion
  defined in~\eqref{eq:A_tilde_general} satisfies
  \begin{equation}   \label{eq:boundary_behavior_mue_tilde}
      \mut(p)=
      \left\{   \begin{array}{llll}
  		  \mu(i)
           &+O\rob{\abs{p-i}}
           & \mbox{ \; if $2\abs{\mu(i)} < \abs{v^{'}(i)}$
                        or $\lambda\kappa_i=0$} \\
  		  \mu(i)+ \frac{v^{'}(i)}{\ln\rob{|p-i|}}
           &+O\rob{\frac{\abs{p-i}}{\ln{\abs{p-i}}}}
           & \mbox{ \; if \ \,$2\mu(i) = v^{'}(i)$ and $\lambda\kappa_i\neq0$} \\
  		  v^{'}(i)-\mu(i)
           &+O\rob{\abs{p-i}}
           & \mbox{ \; if $2\abs{\mu(i)} > \abs{v^{'}(i)}$
                       and $\lambda\kappa_i\neq0$}
  	\end{array}  \right.
  \end{equation}
  as $p\to i$ for $i\in\{0,1\}$.
\end  {lem}
\begin{proof}
  Recall that lemma~\ref{l:boundary_behavior_pi}
  describes the asymptotic behavior of $\pi(\cdot)$ as $p\to i$.
  From this we obtain
  \begin{equation}
      \frac{v(p)\pi^{'}(p)}{\pi(p)}
      =
      \left\{   \begin{array}{lll}
  		  v^{'}(i)\frac{2\mu(i)-v^{'}(i)}{v^{'}(i)}&+O\rob{\abs{p-i}}
          & \mbox{ \; if $2\abs{\mu(i)} < \abs{v^{'}(i)}$
                         or $\lambda\kappa_i=0$} \\
  		  \frac{v^{'}(i)}{\ln{\abs{p-i}}}&+O\rob{\frac{\abs{p-i}}{\ln\abs{p-i}}}
          & \mbox{ \; if \ \,$2\mu(i) = v^{'}(i)$ and $\lambda\kappa_i\neq0$} \\
        0&+O\rob{\abs{p-i}}  
          & \mbox{ \; if $2\abs{\mu(i)} > \abs{v^{'}(i)}$ and $\lambda\kappa_i\neq0$}
  	\end{array}  \right.
  \end{equation}
  as $p\to i$
  for $i\in\{0,1\}$.
  Inserting this into
  \begin{equation}
    \mut(p):=-\mu(p)+\frac{(v\pi)^{'}(p)}{\pi(p)}
      =-\mu(i)+v^{'}(i)+O\rob{\abs{p-i}}+\frac{v(p)\pi^{'}(p)}{\pi(p)}
  \end{equation}
  results in assertion~\eqref{eq:boundary_behavior_mue_tilde}.
  \qed
\end{proof}

\begin{rem} Notice that $\mut(0) < 0$ if $\mu(0) > v'(0)$, so that the diffusive motion
of the time-reversed process need not be confined to $[0,1]$.  Nonetheless,
because the boundary $p = 0$ is inaccessible to the forward diffusion in this
case (i.e., $2\mu(0) > v'(0)$), the fact that the process $\tilde{p}(\cdot)$ immediately
jumps back into $(0,1)$ upon hitting $0$ will ensure that the jump-diffusion
is confined to $[0,1]$.  
\end{rem}

We next show that if a boundary point is accessible to the forward jump-diffusion 
$p(\cdot)$, either diffusively or via a jump, then it must be accessible to the backward diffusive 
motion governed by $\At$.  Recall the scale function $\St$ from~\eqref{eq:def:st_mt}.

\begin{lem}   \label{l:when_accessible}
  Assume~\ref{assumptions}.
  The boundary point $i\in\{0,1\}$ is accessible to the diffusive
  motion governed by $\At$, that is $\St(i)\in\R$, if and only if $i$ is accessible
  to the forward jump-diffusion $p(\cdot)$, that is,
  if $\lambda w_i\not\equiv0$ or $\abs{\mu(i)}<\abs{v^{'}(i)}$.
\end{lem}
\begin{proof}
  W.l.o.g.\ we only prove the case $i=0$.
  According to Lemma 15.6.1 in~\cite{KT2}, the boundary point $0$
  is accessible if and only if $\St(0+)$ is finite.  (This is a special
  case of Feller's boundary classification criteria.)
  Substituting the asymptotic expression for $\mut$ near $p = 0$
  (see Lemma~\ref{l:boundary_behavior_mue_tilde}) into the definition
  of $\St$, we obtain in the case $\lambda\kappa_0>0$
  \begin{equation}  \label{eq:St_asymptotic}
     \St^{'}(p):=\exp\roB{ -\int_{\frac12}^p\frac{2\mut(z)}{v(z)}\,dz }
      \sim
      \left\{   \begin{array}{lll}
  		  p^{-\frac{2\mu(0)}{v^{'}(0)}}
             & \mbox{ \; if $2\abs{\mu(i)} < \abs{v^{'}(i)}$} \\
  		  \frac{1}{p}\frac{1}{\rob{\ln(p)}^2}
             & \mbox{ \; if \ \,$2\mu(i) = {v^{'}(i)}$} \\
  		  p^{-\frac{2v^{'}(0)-2\mu(0)}{v^{'}(0)}}
             & \mbox{ \; if $2\abs{\mu(i)} > \abs{v^{'}(i)}$}
  	\end{array}  \right.
  \end{equation}
  as $p\to 0$. In all three cases, $\St^{'}(\cdot)$ is integrable
  over $(0,\tfrac12]$. The case $\lambda\kappa_0=0$ follows
  from similar arguments.
  \qed
\end{proof}

The following lemma shows that the rate constant $r_i$ (defined in~\eqref{eq:rate}) 
is equal to infinity if the boundary point $i\in\{0,1\}$ is inaccessible to the forward 
diffusive motion.  Therefore $\pt(\cdot)$ jumps whenever it hits $i$ as
$\Lt_i(t)\geq0=R_i$.  In addition, if there are no jumps in the forward process,
then $r_i=0$, and $\pt(\cdot)$ never jumps as $\Lt_i(t)<\infty=R_i$.
%
%
\begin{lem}   \label{l:rate}
  Assume~\ref{assumptions}.
  Then the rate constant $r_i$ used to define the jump times of $\pt(\cdot)$
  satisfies
  \begin{equation}   \label{eq:l:rate}
    r_i:=\lim_{p\to i}\roB{\frac{\mt(p)}{\pi(p)}\lambda \kappa_i}
    \left\{   \begin{array}{ll}
  		  \in(0,\infty)
                   & \mbox{ \; if $2\abs{\mu(i)} < \abs{v^{'}(i)}$
                                 and $\lambda w_i(\cdot)\not\equiv0$ }  \\
        =\infty    & \mbox{ \; if $2\abs{\mu(i)} \geq \abs{v^{'}(i)}$
                              and $\lambda w_i(\cdot)\not\equiv0$}\\
        =0         & \mbox{ \; if $\lambda w_i(\cdot)\equiv0$}
  	\end{array}  \right.
  \end{equation}
  for $i\in\{0,1\}$.
\end  {lem}
\begin{proof}
  W.l.o.g.\ we assume $i=0$ as the case $i=1$ is similar.
  If $\lambda\kappa_0=0$, then $r_0=0$ is trivially correct.
  Assume $\lambda\kappa_0>0$ for the rest of the proof.
  The asymptotic behavior of the scale density $\St^{'}(\cdot)$
  is given in~\eqref{eq:St_asymptotic}.
  From this we derive
  the asymptotic behavior of the speed density $\mt(p)$
  (defined in~\eqref{eq:def:st_mt}) as $p\to 0$
  \begin{equation} \label{eq:asymptotic_behavior_mt}
    \mt(p)=\frac{2}{v(p)\St^{'}(p)}
      \sim
      \left\{   \begin{array}{lll}
  		  p^{\frac{2\mu(0)-v^{'}(0)}{v^{'}(0)}}
             & \mbox{ \; if $2\mu(0) < v^{'}(0)$} \\
  		  \rob{\ln(p)}^2
             & \mbox{ \; if $2\mu(0) = v^{'}(0)$} \\
  		  p^{\frac{v^{'}(0)-2\mu(0)}{v^{'}(0)}}
             & \mbox{ \; if $2\mu(0) > v^{'}(0)$}.
  	\end{array}  \right.
  \end{equation}
  Compare~\eqref{eq:asymptotic_behavior_mt}
  with the boundary behavior of $\pi(\cdot)$
  (see Lemma~\ref{l:boundary_behavior_pi})
  to obtain~\eqref{eq:l:rate}.
  \qed
\end{proof}

%
%
\section{The generator of the time-reversed process}%
\label{sec:the generator of the time-reversed process}

In this section we identify the generator of the time-reversed process and
show that this operator satisfies the duality condition given in~\eqref{eq:G}.  That this
operator is also the generator of the jump-diffusion process $\pt(\cdot)$ described
in Section~\ref{sec:main_result} will be established in the final two sections of 
the paper.

The following notation will be needed to formulate the generator of the 
time-reversed process.  If $\nu(dp)=f(p)dp$ is a measure on $[0,1]$ with 
continuous density $f\colon(0,1)\to(0,\infty)$ with respect to Lebesgue measure, 
then we write
\begin{equation}
  \roB{D_\nu \psi}(p):=\lim_{y\to p}\frac{\psi(y)-\psi(p)}{\int_p^y f(q)dq}
  \qquad\fa p\in(0,1)
\end{equation}
whenever this limit exists in $\R$ and denote by
\begin{equation}  \begin{split}
  \Dom(D_\nu):=
  \Bigl\{
    \psi\in\C\rub{[0,1]}\colon& D_\nu \psi(p)\text{ exists}\fa p\in(0,1),
    \ D_\nu  \psi\text{ is continuous}\\
    &\text{in }(0,1),\,D_\nu \psi(0+)\text{ and }D_\nu \psi(1-)
    \text{ exist in }\R
  \Bigr\}
\end{split}     \end{equation}
the subset of functions which are mapped to continuous functions on $[0,1]$.
Note that $\psi\in\C^1\rob{(0,1)}$ and $D_\nu \psi(p)=\tfrac{\psi^{'}(p)}{f(p)}$, $p\in(0,1)$,
for every $\psi\in\Dom(D_\nu)$.  For $\psi\in\Dom(D_\nu)$ the definition of $D_\nu$ 
extends to the boundary via $D_\nu\psi(0):=D_\nu\psi(0+)$ and $D_\nu\psi(1):=D_\nu\psi(1-)$.
In this notation, the generator $\At$ of the backward diffusive motion reads as
\begin{equation*}
  \rub{\At\psi}(p)=\tfrac12v(p)\psi^{''}(p)+\mut(p)\psi^{'}(p)
     =\tfrac12v(p)\St^{'}(p)
       \ruB{  \frac{1}{\St^{'}}\psi^{''}
             +\frac{1}{\St^{'}}\frac{2\mut}{v}\psi^{'}
           }(p)
     =D_{\mt}D_{\St}\psi(p)
\end{equation*}
for every $\psi\in\C^2\rub{[0,1]}$.
The following set will be a core for the generator of the time-reversed process
\begin{equation}  \begin{split}  \label{eq:coreGent}
  \coreGent:=
    \biggl\{
      \psi\in\Dom\rub{D_\St}\colon&
      D_\St\psi\in\Dom\rub{D_\mt } \text{ and for }i\in\{0,1\}
    \\
    \lim_{p\to i}&\rub{\tfrac12 v\pi\psi^{'}}(p)
           =(-1)^{i+1}\int_0^1 \rub{\psi(p)-\psi(i)}\lambda w_i(p)\pi(p)\,dp
    \biggr\}.
\end  {split}     \end  {equation}
The following lemma asserts that the restriction of $D_{\mt}D_{\St}$ to $\coreGent$
extends to a strong generator of a Markov process.  Indeed, this can be deduced
from Theorem II.4 of Mandl (1968)\nocite{Mandl} which shows that the restriction 
of $D_{\mt}D_{\St}$ to 
\begin{equation}  \begin{split}  \label{eq:restricted_domain}
  \biggl\{\psi\in\Dom\rob{D_\St}\colon& D_\St\psi\in\Dom(D_\mt)
       \text{ and if }i\in\{0,1\}\text{ is accessible:}\\
  &\chi_i \psi(i)+\int_0^1\frac{\psi(i)-\psi(p)}{\abs{p_i(i)-p_i(p)}}  d q_i(p)
     +\eta_i D_\mt D_\St\psi(i)=0\biggr\}
\end{split}     \end{equation}
is the strong generator of a Feller semigroup if $\chi_i$ and $\eta_i$ are both 
non-negative, if $q_i$ is a non-decreasing function on $[0,1]$, and if $p_i$ is 
continuous, non-decreasing and equal to $\St$ in a neighborhood of $i$, $i=0,1$.  
Only the case $0=\chi_0 =\chi_1=\eta_0=\eta_1=q(1-)-q(0+)\neq q(1)-q(0)$ is 
excluded.  The quotient in~\eqref{eq:restricted_domain} is to be interpreted as 
$(-1)^{i+1}D_\St\psi(i)$ for $p=i\in\{0,1\}$ and the integral with respect to $dq_i(p)$ 
denotes the Lebesgue-Stiltjes integral with respect to $q_i$.
%
%
\begin{lem}   \label{l:Gent_is_generator}
  Assume~\ref{assumptions}.
  The restriction of $D_\mt D_\St$ to the set $\coreGent$
  extends to a strong generator
  of a Markov process.
\end  {lem}
\begin{proof}
  By Theorem II.4 in~\cite{Mandl}, it suffices to prove that $\coreGent$
  is of the form~\eqref{eq:restricted_domain}.
  According to Lemma~\ref{l:when_accessible}, the boundary point
  $i\in\{0,1\}$ is accessible to the diffusion governed by $D_\mt D_\St$
  if and only if $\lambda w_i(\cdot)\not\equiv 0$ or $\abs{\mu(i)}<\abs{v^{'}(i)}$.
  First we show that the condition in~\eqref{eq:coreGent} is trivial
  if $i$ is inaccessible, that is, we show $\lim_{p\to i}\rob{v \pi \psi^{'}}(p)=0$ 
  for every $\psi\in \Dom(D_\St)$.  Suppose that
  \begin{equation}
    C:=\lim_{p\to i}\rob{v\pi\psi^{'}}(p)>0
  \end  {equation}
  for some function $\psi\in\Dom\rob{D_\St}$.
  By Lemma~\ref{l:boundary_behavior_pi}, $\pi(p)$ is bounded above
  by $-\Cb \ln\rob{\abs{p-i}}$ in a neighborhood of $i$ for
  some $\Cb>0$. Thus, in a neighborhood of $i$,
  \begin{equation}
    \psi^{'}(p)\geq\frac{-C}{2 v(p)\pi(p)}
         \geq \frac{C}{4\Cb \abs{v^{'}(i)}\abs{p-i}\ln\rob{\abs{p-i}}}.
  \end  {equation}
  Integrating over $[\tfrac12,p]$ implies that $\psi(p)$ is bounded below by
  $\tfrac{C}{4\Cb \abs{v^{'}(i)}}\ln\rob{-\ln(\abs{p-i})}$ as $p\to i$ which contradicts 
  $\psi\in\C(I)$.  An analogous argument applies to the case $C<0$.

  Now, let $i$ be accessible to the forward process, that is, $\lambda \kappa_i>0$ 
  or $\abs{\mu(i)}<\abs{v^{'}(i)}$.
  Note that $\St(\cdot)$ is a bounded continuous non-decreasing function
  in a neighborhood of $i$.  Choose $\chi_i,\eta_i=0$ and $p_i(p):=\St(p)$.
  Furthermore, let
  \begin{equation}
    q_i(p):=\int_0^p \abs{p_i(i)-p_i(y)}w_i(y)\pi(y)\,dy
     +c_0\1_{(0,1]}(p)\1_{i=0} +c_1\1_1(p)\1_{i=1}
  \end  {equation}
  where $c_0,c_1\in[0,\infty)$ are to be chosen later.  Note that $q_i$ is bounded  
  and that $dq_i$ puts mass $c_i$ on the point $i$.  With these definitions, the 
  condition in~\eqref{eq:restricted_domain} takes the form
  \begin{equation}   \label{eq:takes_the_form}
    \int_0^1 \rob{\psi(i)-\psi(p)}w_i(p)\pi(p)\,dp=(-1)^i D_\St \psi(i) c_i.
  \end  {equation}
  It remains to choose $c_i\in[0,\infty)$ such that
  \begin{equation}  \label{eq:find_c}
     D_\St \psi(i)c_i =\lim_{p\to i}(\tfrac12 v\pi\psi^{'})(p)
  \end{equation}
  for every $\psi\in \coreGent$.
  Using the boundary behavior~\eqref{eq:St_asymptotic}
  of $\St(\cdot)$ and the asymptotic
  behavior~\eqref{eq:boundary_behavior_pi} of $\pit(\cdot)$,
  we arrive at
  \begin{equation*}
    \rob{v\pi\psi^{'}}(p) \sim
       \abs{p-i}\mal \pi(p)
       \mal \St^{'}(p)\mal D_{\St}\psi(p)
    \sim
    \left\{   \begin{array}{ll}
  		 D_\St\psi(p) & \mbox{if $2\abs{\mu(i)} < \abs{v^{'}(i)}$}  \\
  		 D_\St\psi(p)\frac{-1}{\ln\rob{\abs{p-i}}}
                    & \mbox{if \ \,$2\mu(i) = v^{'}(i)$}  \\
  		 D_\St\psi(p)\abs{p-i}^{\frac{2\mu(i)}{v^{'}(i)}-1}
  		              & \mbox{if $2\abs{\mu(i)} > \abs{v^{'}(i)}$}
  	\end{array}  \right.
  \end  {equation*}
  as $p\to i$. This shows that~\eqref{eq:find_c} holds with some
  constant $c_i\in[0,\infty)$.
  \qed
\end  {proof}

%
%
\begin{lem}   \label{l:generator_backward_process}
  Assume~\ref{assumptions}.
  Let the process $\rob{p(t)\colon t\geq0}$ be in equilibrium.
  Then the time-reversed process $\rob{\pb_t\colon t\geq0}$ exists, that is,
  there exists a process $\rob{\pb_t\colon t\geq0}$ satisfying
  \begin{equation}
    \rob{\pb(t)\colon t\leq T}\eqd\rob{p(T-t)\colon t\leq T}\qquad\fa T\geq0.
  \end  {equation}
  In addition, $\coreGent$ is a core for the generator $\Gt$
  of $\rob{\pb_t\colon t\geq0}$ and
  \begin{equation} \label{eq:Gent}
    \Gt \psi= \tfrac12 v\psi^{''}+\mut \psi^{'}=D_\mt D_\St \psi
    \qquad\fa \psi\in\coreGent.
  \end  {equation}
\end  {lem}
\begin{proof}
  Let $\Gt$ be the closure of the operator
  defined in~\eqref{eq:Gent}.
  By Lemma~\ref{l:Gent_is_generator},
  $\Gt$ is the strong generator of a Markov
  process $\rob{\pb_t\colon t\geq0}$.
  Recall the generator $G$ of $p(\cdot)$ from~\eqref{eq:G_general}.
  We will prove that $\Gt$ is the adjoint operator of $G$
  with respect to the invariant measure $\pi(p)dp$.
  Let $0<\eps<\tfrac12$, $\phi\in\C^2\rub{[0,1]}$ and
  $\psi\in\coreGent$.
  The functions $v,\mu,\phi,\psi$ and $\pi$ are $\mathbf{C}^2$
  in $[\eps,1-\eps]$.
  Integration by parts yields
  \begin{equation}  \begin{split} \label{eq:int_by_parts2}
    \int_\eps^{1-\eps}&G\phi\mal\psi\mal\pi \,dp
       -\phi(0)\int_\eps^{1-\eps}\lambda w_0\mal\psi\mal\pi \,dp
       -\phi(1)\int_\eps^{1-\eps}\lambda w_1\mal\psi\mal\pi \,dp\\
     &=\int_\eps^{1-\eps}\phi^{''}\mal (\tfrac12 v\psi\pi) +\phi^{'}\mal(\mu\psi\pi)
       -\lambda\phi\psi\pi \,dp\\
     &=[\phi^{'}\tfrac12 v\psi\pi]_\eps^{1-\eps}
       + [\phi\mal\rob{\mu\psi\pi-(\tfrac12 v\psi\pi)^{'}}]_\eps^{1-\eps}\\
     &\qquad\qquad\qquad\;
          +\int_\eps^{1-\eps}\phi\mal\eckB{(\tfrac12 v\psi\pi)^{''}
          -(\mu\psi\pi)^{'} -\lambda\psi\pi}\,dp.
  \end  {split}     \end  {equation}
  As $\pi$ satisfies~\eqref{eq:pi}, we see that
  \begin{equation}
    (\tfrac12 v\psi\pi)^{''}-(\mu\psi\pi)^{'}-\lambda\psi\pi
    =\tfrac12 v\psi^{''}\pi+\rob{\frac{(v\pi)^{'}}{\pi}-\mu}\psi^{'}\pi
    =D_{\mt}D_{\St}\psi\mal \pi.
  \end  {equation}
  The functions $G \phi$, $\psi$, $w_0$, $w_1$, $\phi$ and 
  $D_{\mt}D_{\St}\psi$ are bounded and $\pi$ is integrable.
  Hence, we may apply the dominated convergence theorem to the
  integrals in~\eqref{eq:int_by_parts2} as $\eps\to0$.
  This also proves that the limits of the
  boundary terms in~\eqref{eq:int_by_parts2} exist as $\eps\to0$.
  Thus letting $\eps\to0$ in~\eqref{eq:int_by_parts2}, we obtain
  \begin{equation}  \begin{split} \label{eq:Gent_equals_Atranspose}
    \lefteqn{\int_0^1G\phi\mal\psi\pi\,dp
    -\int_0^1 \phi\mal D_{\mt}D_{\St}\psi\mal \pi\,dp}\\
    &=\phi(1)\robb{\psi(1)\rob{\mu\pi-(\tfrac12 v\pi)^{'}}(1-)
       -\lim_{p\to 1}\rob{\tfrac12 v\pi\psi^{'}}(p)+\int_0^1\lambda w_1\mal\psi\mal\pi\,dp}\\
    &\quad-\phi(0)\robb{\psi(0)\rob{\mu\pi-(\tfrac12 v\pi)^{'}}(0+)
       -\lim_{p\to 0}\rob{\tfrac12 v\pi\psi^{'}}(p)-\int_0^1\lambda w_0\mal\psi\mal\pi\,dp}\\
    &=0
  \end  {split}     \end  {equation}
  for all $\phi\in\mathbf{C}^2$ and $\psi\in\coreGent$.
  The last equality follows from~\eqref{eq:pi} and from $\psi\in\coreGent$.
  This proves that $G$ and $\Gt$ are adjoint to each other.
  Consequently, the semigroups of $\rob{p(t)\colon t\geq0}$ and of
  $\rob{\pb_t\colon t\geq0}$ are adjoint to each other.  According to~\cite{Nel58},
  this implies that the Markov process $\rob{\pb(t)\colon t\geq0}$ 
  associated with $\Gt$ has the same law as the time-reversed process
  of $\rob{p(t)\colon t\geq0}$.
  \qed
\end  {proof}

%
%
\noindent
\section{The local time process}%
\label{sec:The local time process}

This section describes some properties of the local time process of $\pt(\cdot)$.
First we show existence.
Recall the scale function $\St$ and the speed measure $\mt$
from~\eqref{eq:def:st_mt}.
%
%
\begin{lem}  \label{l:local_time_process}
  Assume~\ref{assumptions}.
  Then there exists a unique, non-negative process
  \begin{equation}
    \rob{\Lt_p(t)\colon t\geq0,\,p\in[0,1]}
  \end{equation}
  which is almost
  surely continuous in $(t,p)$ and which satisfies
  \begin{equation}  \label{eq:l:occupation_time_formula_Lt}
    \int_0^t f\rob{\pt(u)}\,du=\int_0^1 f(p)\Lt_p(t)\,\mt(dp)\qquad
    \fa t\geq0
  \end  {equation}
  for all measurable $f\colon[0,1]\to[0,\infty)$ almost surely.
  In addition, if $2\abs{\mu(i)}\geq\abs{v^{'}(i)}$ for $i\in\{0,1\}$,
  then $\Lt_i(\cdot)\equiv 0$ almost surely.
\end  {lem}
\begin{proof}
  Let $0=:\tau_0\leq\tau_1<\tau_2<\cdots$ be the jump times of
  $\pt(\cdot)$.
  Then, by construction of $\pt(\cdot)$,
  $\rob{\pt(t+\tau_{n-1})\colon 0\leq t<\tau_n-\tau_{n-1}}$,
  $n\in\Nat_{\geq1}$, are independent diffusions governed by $\At$.
  It is well-known that $\pt(\cdot+\tau_{n-1})$
  can be written in terms of a Brownian motion as follows.
  Let $\curlb{B^{y,n}_{\cdot}\colon y\in\R, n\in\Nat}$
  be a family of independent standard Brownian motions with
  $B^{y,n}_{0}=y$.
  Denote by $\rob{L_x^{B^{y,n}}(t)\colon t\geq0, x\in\R}$ the
  local time process of $B^{y,n}_{\cdot}$, see e.g.\ Section 2.8
  in~\cite{ItoMcKean}, and define
  \begin{equation}
    \xi_t^{(n)}:=\int_0^1 L_{\St(p)}^{B^{\St(\pt(\tau_{n-1})),n}}(t)\,\mt(dp)
    \qquad t\geq0.
  \end{equation}
  Then a version of $\pt(\cdot+\tau_{n-1})$ is given by
  \begin{equation}
    \pt(t+\tau_{n-1})=\St^{-1}\robb{
       B_{\rob{\xi_t^{(n)}}^{-1}}^{\St\rob{\pt(\tau_{n-1})},n}
       }\quad 0\leq t<\tau_n-\tau_{n-1}.
  \end{equation}
  Inserting this into the occupation time formula of the Brownian motion,
  a short calculation (see e.g.\ Section 5.4 in~\cite{ItoMcKean}) shows that
  \begin{equation} \label{eq:occupation_time_Y}
    \int_0^t f\rob{\pt(r+\tau_{n-1})}\,dr
    =\int_0^1 f(p)\Lh_p^{ (n) }(t)\mt(dp)\qquad 0\leq t<\tau_n-\tau_{n-1},
  \end{equation}
  where the local time process of $\pt(\cdot+\tau_{n-1})$ with respect
  to the speed measure is
  \begin{equation}
    \Lh_p^{ (n) }(t)
    :=L_{\St(p)}^{B^{\St(\pt(\tau_{n-1})),n}}\roB{(\xi_t^{(n)})^{-1}}
    \qquad0\leq t<\tau_n-\tau_{n-1},\,p\in[0,1], n\in\Nat_{\geq1}.
  \end{equation}
  Now we put the independent path segments together by defining
  \begin{equation}   \label{eq:def:Lt}
    \Lt_p(t)
    :=\Lh_p^{ (n) }\rob{t-\tau_{n-1}}
    +\sum_{k=1}^{n-1} \Lh_p^{ (k) }\rob{ \tau_k-\tau_{k-1} }
    \qquad
    \text{if }\tau_{n-1}\leq t<\tau_n.
  \end{equation}
  It is easy to use~\eqref{eq:occupation_time_Y} to show that 
  $\rob{\Lt_p(t)\colon t\geq0,p\in[0,1]}$ satisfies~\eqref{eq:l:occupation_time_formula_Lt}.
  Uniqueness follows from standard arguments.

  If $2\abs{\mu(i)}\geq\abs{v^{'}(i)}$, then $\pt(\cdot)$ jumps into $(0,1)$ as soon 
  as it hits the boundary and we conclude that $\pt(t)\neq i$ for all $t\geq0$.
  Thus the local time at this boundary point is identically zero.
  \qed
\end{proof}

In the next section, we will need to be able control the second moment of the local time
of the time-reversed jump-diffusion at a boundary point.  We first prove the following
estimate concerning the local time of a standard Brownian motion.

%
%
\begin{lem}  \label{l:bound_on_LBsquare}
  Let $\eps,\dl>0$ and let $\rob{B_t\colon t\geq0}$ be a standard
  Brownian motion with local time $L_x^B(\cdot)$ at $x\in\R$.
  Suppose that the function $\Sb\colon[0,\delta]\to\R$ is 
  non-decreasing. Define $\Delta:=\Sb(\delta)-\Sb(0)$ and 
  \begin{equation}
    \zeta_t:=\int_{0}^\dl\frac{1}{\eps} L_{\Sb(y)}^B (t)\,dy\quad t\geq0.
  \end  {equation}
  Then, for each $m>0$, there exists a constant $C_m$ independent of 
  $\eps$ and of $\dl$ such that
  \begin{equation}  \label{eq:bound_on_LBsquare}
    \Expectation^{\Sb(p)} \eckbb{\roB{L_{\Sb(p)}^B \rob{\zeta^{-1}_t}
             }^m}
    \leq C_m\ruB{\eps t}^{\frac{m}{2}}
          \eckB{\ruB{\frac{\eps t}{\Delta^2}}^{\frac{m}{2} }
                +1
               }
  \end  {equation}
  for all $p\in[0,\delta]$ and $t\geq 0$.
\end  {lem}
\begin{proof}
  Inequality~\eqref{eq:bound_on_LBsquare} is trivial if $\Delta=0$
  or $t=0$, so we may and will assume that $\Delta>0$, $t>0$.
  Fix $p\in[0,\delta]$.
  The left-hand side of~\eqref{eq:bound_on_LBsquare} does not depend
  on the value of $\Sb(p)$, so we will also assume w.l.o.g.\ that $\Sb(p)=0$.
  Denote by $B^*_t:=\max\{B_s\colon s\leq t\}$ and by 
  $\abs{B}^*_t:=\max\{\abs{B_s}\colon s\leq t\}$ the process of the maximum 
  and the process of the absolute maximum, respectively.
  Define  $Z_t:= \Sb^{-1}\rob{B_{\zeta^{-1}_t}}$ for $t\geq0$. According to
  Section~5.4 in~\cite{ItoMcKean}, $L_p^Z(t):=L_{0}^B\rob{\zeta^{-1}_t}$
  is the local time of $\rob{Z_t\colon t\geq0}$ at $p$.
  The process $\rob{Z_t\colon t\geq0}$ is equal in distribution to the process
  $\rob{\Sb^{-1}\rob{B_{\eps t}}\colon t\geq0}$
  reflected at $0$ and at $\dl$.
  Another way to construct $\rob{\Sb(Z_t)\colon t\geq0}$ is to take the
  path of $\rob{B_{\eps t}\colon t\geq0}$
  and to identify each $x\in[\Sb(0),\Sb(\dl)]$
  with the set $\{x+2\Delta z,2\Sb(\dl)-x+2\Delta z\colon z\in\Z\}$.
  Thus the local time $L_p^Z(t)$ of $\rob{Z_t\colon t\geq0}$ in $p$
  is equal in distribution to the sum of
  $L_{2\Delta z}^B\rob{\eps t} + L_{2\Delta z+2\Sb(\dl)}^B\rob{\eps t}$
  over $z\in\Z$.  Note that $L_{x}^B\rob{\eps t}=0$ almost surely
  on the event $\{\abs{B}_{\eps t}^*<\abs{x}\}$, $x\in\R$.
  In addition, note that convexity of $0\leq x\mapsto x^m$
  implies $k^{m-1}(a_1+\cdots+a_k)^m\leq (a_1^m+\cdots+a_k^m)$
  for $a_1,\ldots,a_k\geq0$, $k\in\Nat$.  Therefore,
  \begin{equation}  \begin{split}  \label{eq:estimate_LZ}
    \lefteqn{\Expectation^p\eckB{\rub{L_p^Z(t)}^m}}\\
    &=\Expectation^{0}
      \eckbb{
           \sum_{k\in 2\Delta\Nat_{\geq0}}
           \!\!\!\!\1_{[k,k+2\Delta)}\rob{\abs{B}_{\eps t}^*}
           \biggl(\sum_{z=-\frac{k}{2\Delta}}^{\frac{k}{2\Delta}}
                 \sum_{x\in\{0,2 \Sb(\dl)\}}
             L_{2\Delta z+x}^B\rob{\eps t}\biggr)^m
      }\\
    &\leq\Expectation^{0}
      \eckbb{
           \sum_{k\in 2\Delta\Nat_{\geq0}}
           \!\!\!\!\1_{[k,k+2\Delta)}\rob{\abs{B}_{\eps t}^*}
           \mal\ruB{\frac{2k}{\Delta}+2}^{m-1}
           \!\!\!\!\sum_{z=-\frac{k}{2\Delta}}^{\frac{k}{2\Delta}}
           \sum_{x\in\{0,2 \Sb(\dl)\}}
             \!\!\ruB{L_{2\Delta z+x}^B\rob{\eps t}}^m
      }.
  \end  {split}     \end  {equation}
  Use the strong Markov property (e.g.\ Proposition 2.6.17
  in~\cite{KaSh}) to restart the Brownian motion at the
  first hitting time
  of ${2\Delta z}$ and
  of ${2\Delta z+2\Sb(\dl)}$, respectively.
  Thus the left-hand side of~\eqref{eq:estimate_LZ} is bounded above by
  \begin{equation}  \begin{split}  \label{eq:asdf}
    &\sum_{k\in 2\Delta\Nat_{\geq0}}
           \P^{0}\ruB{\abs{B}_{\eps t}^*\in [k,k+2\Delta)}
           \mal\ruB{\frac{2k}{\Delta}+2}^{m-1}
           2\sum_{i=-\frac{k}{2\Delta}}^{\frac{k}{2\Delta}}
             \Expectation^{0}\eckB{\ruB{L_0^B\ro{\eps t}}^m}\\
    &\leq
        \Expectation^{0}
            \eckbb{
            \Bigl(\frac{2\abs{B}_{\eps t}^*}{\Delta}+2\Bigr)^{m}
            }
        \Expectation^{0}\eckB{\roB{L_0^B\ro{\eps t}}^m}.
  \end{split}     \end{equation}
  Note that 
  $2 L_0^B(t)$ and $B_t^*$ are equal in distribution,
  see e.g.\ Theorem 3.6.17 in~\cite{KaSh}.
  Therefore the left-hand side of~\eqref{eq:estimate_LZ} is bounded above by
  \begin{equation}  \begin{split}
    &2^{m}
            \eckbb{
            \Expectation^{0} \eckB{\Bigl(\frac{\abs{B}_{\eps t}^*}{\Delta}\Bigr)^{m}}
            +1
            }
      \Expectation^{0}\eckB{\roB{\frac{1}{2}B_{\eps t}^*}^m}
        \\
    &\leq
        \eckB{K_{\frac{m}{2}}\ruB{\frac{\eps t}{\Delta^2}}^{\frac{m}{2}}
              +1}
        K_{\frac{m}{2}}\ruB{\eps t}^{\frac{m}{2}}
  \end  {split}     \end  {equation}
  where $K_{m/2}\geq1$ is a suitable constant which is independent
  of $\Delta, \eps$ and $t$.
  The last step follows from the
  Burk\-hol\-der-Da\-vis-Gun\-dy inequality,
  see e.g.\ Theorem~3.3.28 in~\cite{KaSh}.
  Therefore \eqref{eq:bound_on_LBsquare} holds with
  $C_m:=K_{\frac{m}{2}}^2$.
  \qed
\end  {proof}

In the proof of Theorem~\ref{thm:main_result}, we will need to exploit the
fact that, in the $L^2$ sense, the local time $\rob{\Lt_i(t)\colon t\geq0}$ at a 
boundary point $i\in\{0,1\}$ of the backwards process started at $i$ decreases 
to zero faster than $\sqrt{t}$ as $t\to0$.  This might be surprising as one can 
show that
\begin{equation}
  \Expectation^0\eckB{\rob{L_0^B(t)}^2}\sim t\qqastO.
\end{equation}
However, the infinitesimal variance $v(\cdot)$ is zero in $i$.
Thus, informally speaking, the diffusion governed by $\At$
is pushed away from zero almost deterministically at rate $\mut(i)>0$
if the boundary point is accessible at all.
%
%
\begin{lem}    \label{l:Lt_is_o_of_sqrt_t}
  Assume~\ref{assumptions}.
  Then the local time at the boundary satisfies
  \begin{equation}   \label{eq:Lt_is_o_of_sqrt_t}
    \lim_{t\to0}\frac1t \Expectation^i \eckbb{\roB{\Lt_i(t)}^2}=0
  \end  {equation}
  for $i\in\{0,1\}$.
\end  {lem}
\begin{proof}
  If $2\abs{\mu(i)}\geq\abs{v^{'}(i)}$, then Lemma~\ref{l:local_time_process}
  tells us that $\Lt_i(t)=0$, which implies the assertion in this case.
  For the rest of the proof assume that $2\abs{\mu(i)}<\abs{v^{'}(i)}$.
  W.l.o.g.\ we assume that $i=0$ as the case $i=1$ is similar.
  To begin, we prove that~\eqref{eq:Lt_is_o_of_sqrt_t}
  holds with $\Lt_0(\cdot)$ replaced by $\Lh_0^{(1)}(\cdot)$.
  For $\dl\in(0,1)$ define
  \begin{equation}
    \uline{\mt}(\dl)=\inf_{x\leq \dl}\mt(x).
  \end  {equation}
  The asymptotic behavior~\eqref{eq:asymptotic_behavior_mt}
  of $\mt(\cdot)$ implies $\lim_{\dl\to0}\uline{\mt}(\dl)=\lim_{p\to 0}\mt(p)=\infty$.
  Recall that $B^{y,1}_\cdot$ is a standard Brownian motion started
  at $B_0^{y,1}=y$. Observe that
  \begin{equation}
    \xi_t^{(1)}:=
    \int_0^1 L_{\St(y)}^{B^{\St(0),1}}(t) \mt(p)\,dp
    \geq\uline{\mt}(\dl)\int_0^\dl L_{\St(y)}^{B^{\St(0),1}}(t)\,dp
    =:\zeta_t\quad\fa t\geq0.
  \end  {equation}
  Using $\rob{\xi_t^{(1)}}^{-1}\leq\zeta^{-1}$ we obtain
  an upper bound for $\Lh_0^{(1)}(\cdot)$ as follows
  \begin{equation}  \begin{split}  \label{eq:LB_tau_LB_sigma}
    \lefteqn{
    \frac1t\Expectation^0\eckB{\rub{\Lh_0^{(1)}(t)}^2}
    =\frac1t\Expectation\eckB{
       \ruB{L_{\St(0)}^{B^{{\St(0)},1}} \roB{\rob{\xi_t^{(1)}}^{-1}}
           }^2                   }
    }\\
    &\leq\frac1t\Expectation\eckB{
       \ruB{L_{\St(0)}^{B^{{\St(0)},1}} \rob{\zeta_t^{-1}}
           }^2                   }
    \leq\frac{C_2}{\uline{\mt}(\dl)}
         \eckbb{\frac{t}{\uline{\mt}(\dl)\rob{\St(\delta)-\St(0)}^2}
                +1}\\
    &\lratO \frac{C_2}{\uline{\mt}(\dl)}\lradlO 0
  \end  {split}     \end  {equation}
  for some constant $C_2$ which is independent of $t$ and $\dl$.
  The last inequality is Lemma~\ref{l:bound_on_LBsquare}.

  Now we come to the local time process $\Lt_0(\cdot)$.
  Recall $r_0$, $R_0$ from Section~\ref{sec:main_result}
  and let $\tau_1$ be the first jump time of $\pt(\cdot)$
  from the boundary point $0$.
  The local time $\Lh_0(t)$ converges to zero almost surely
  as $t\to0$.
  By the theorem of dominated convergence, this implies 
  $\Expectation^0\Lh_0(t)\to0$  as $t\to0$.
  Thus there exists a $t_0\geq0$ such that
  $r_0\Expectation^0 \rob{\Lh_0^{(1)}(t)}\leq 1/4$ for
  all $t\leq t_0$.
  Then we obtain from the definition~\eqref{eq:def:Lt}
  of $\Lt_0(\cdot)$ and from the Markov property
  \begin{equation*}  \begin{split}
    \Expectation^0\eckbb{\roB{\Lt_0(t)}^2\1_{\tau_1<t}}
    &\leq
      2\Expectation^0\eckbb{\roB{\Lh_0^{(1)}(t)}^2}
      +2\Expectation^0
        \eckbb{
           \roB{\Lt_0(t)-\Lh_0^{(1)}(t)}^2
           \1_{\Lh_0^{(1)}(t)\geq R_0}
        }\\
    &\leq
      2\Expectation^0\eckbb{\roB{\Lh_0^{(1)}(t)}^2}
      +2\Expectation^{\Law{\pt(\tau_1)}}
        \eckbb{
           \roB{\Lt_0(t)}^2
        }
      \Expectation^0\eckbb{1-e^{-r_0\Lh_0^{(1)}(t)}}
      \\
    &\leq
      2\Expectation^0\eckbb{\roB{\Lh_0^{(1)}(t)}^2}
      +2\Expectation^{0}
        \eckbb{
           \roB{\Lt_0(t)}^2
        }
      \frac14
  \end{split}     \end{equation*}
  for all $t\leq t_0$.
  Using this estimate we get for $t\leq t_0$
  \begin{equation}  \begin{split}
     \Expectation^0\eckbb{\roB{\Lt_0(t)}^2}
     &\leq\Expectation^0\eckbb{\roB{\Lh_0^{(1)}(t)}^2}
         +\Expectation^0\eckbb{\roB{\Lt_0(t)}^2\1_{\tau_1<t}}\\
     &\leq3\Expectation^0\eckbb{\roB{\Lh_0^{(1)}(t)}^2}
         +\frac{1}{2}\Expectation^0\eckbb{\roB{\Lt_0(t)}^2}.
  \end{split}     \end{equation}
  Therefore
  \begin{equation}
    \limtO\frac{1}{t} \Expectation^0\eckbb{\roB{\Lt_0(t)}^2}
    \leq
    6\;\limtO\frac{1}{t} \Expectation^0\eckbb{\roB{\Lh_0^{(1)}(t)}^2}
    =0
  \end{equation}
  where the last equality is~\eqref{eq:LB_tau_LB_sigma}.
  \qed
\end  {proof}

%
%
\noindent
\section{The backward process}%
\label{sec:the_backward_process}

In Section~\ref{sec:the generator of the time-reversed process}, we identified the 
generator $\Gt$ of the time-reversed process (see Lemma \ref{l:generator_backward_process}.)  
However, while it is clear that the boundary behavior of this process must be prescribed 
by the domain of the generator, it is difficult to see how a qualitative description of 
the process can be deduced from the analytical condition that defines this domain.
To address this issue, we show in this section that the process $\pt(\cdot)$ defined in 
Section~\ref{sec:main_result} also has generator $\Gt$.  This confirms the heuristic 
arguments given in the introduction and shows that the time-reversed process is a 
jump-diffusion process whose jump times depend on the local time process constructed 
in the preceding section.

The proof of Theorem~\ref{thm:main_result} is based on the It\^o-Tanaka formula for 
semimartingales which involves the semimartingale local time process.  Because this
local time differs by a scalar factor from the local time process introduced in Section 6 
(see Eq.~\eqref{eq:Lt_Lb} below), we have restated the semimartingale It\^o-Tanaka formula
in terms of $\Lt_{p}(\cdot)$.  This is done in the following lemma.
%
%
\begin{lem}  \label{l:Ito_Tanaka_Y_tilde}
  Assume~\ref{assumptions}.
  Let $\rob{\Yt(t)\colon t\geq0}$ be a diffusion corresponding to the generator $\At$
  defined in (8).  Then for each $\psi\in\coreGent$
  \begin{equation}  \begin{split} \label{eq:Ito_Tanaka_Y_tilde}
    \psi(\Yt(t))-\psi(\Yt(0))
    &=\int_0^t D_\mt D_\St \psi(\Yt(u))\,du
      +\int_0^t \psi^{'}(\Yt_u)\sqrt{v(\Yt_u)}dB_u\\
    &\qquad  + \frac{1}{2} \Lt_{0}(t)\mal D_{\St}\psi(0)
             - \frac{1}{2} \Lt_{1}(t)\mal D_{\St}\psi(1)
  \end  {split}     \end  {equation}
  for all $t\geq0$ almost surely.
\end  {lem}
\begin{proof}
  Fix $\psi\in\coreGent$.
  We approximate $\psi$ with suitable functions and apply the semimartingale 
  It\^o-Tanaka formula.  Denote by $\rob{\Lb_p(t)}_{t\geq0,p\in[0,1]}$ the
  semimartingale local time process of $\ro{\Yt(t)}_{t\geq0}$.  We 
  remark that, in general, this local time process is distinct from the
  local time, $\Lt_p(\cdot)$, introduced in the preceding section.
  By Theorem 3.7.1 in~\cite{KaSh} we may and we will assume
  that $\Lb_p(t)$ is continuous in $t$ and \cadlag\ in $p$.
  The occupation time formula (Theorem 3.7.1 in~\cite{KaSh}) states that
  \begin{equation} \label{eq:occupation_time_formula_Lb}
    \int_0^t \psi\rob{\Yt(u)}v\rob{\Yt(u)}\,du
    =2\int_0^1 \psi(p)\Lb_p(t)\,dp,\qquad t\geq0,
  \end  {equation}
  almost surely.  Let $f$ be a continuous function which is
  $\C^2$ except in $\{a_1,\ldots, a_n\}\subset[0,1]$
  and which admits finite limits $f^{'}(a_k+)$ and $f^{'}(a_k-)$, $k=1,\ldots,n$.
  Then the It\^o-Tanaka formula for continuous semimartingales
  (see Theorem 3.7.1 and Problem 3.6.24 in~\cite{KaSh}) states that
  \begin{equation}  \begin{split}   \label{eq:ito_tanaka_semimartingale}
    \lefteqn{f(\Yt_t)-f(\Yt_0)}\\
    &=\int_0^t f^{'}(\Yt_u)\mut(\Yt_u)\,du
      +\int_0^t f^{''}(\Yt_u)\tfrac{1}{2}v(\Yt_u)\,du
      +\int_0^t f^{'}(\Yt_u)\sqrt{v(\Yt_u)}dB_u\\
    &\quad+\sum_{k=1}^n \Lb_{a_k}(t)
        \eckB{ f^{'}\rob{a_k+} - f^{'}\rob{a_k-} }
  \end  {split}     \end  {equation}
  almost surely.

  For every $n\in\Nat$,
  let $\psi_n$ be a continuous function which is equal to $\psi$ in
  $(\tfrac1n,1-\tfrac{1}{n})$ and which is constant both in $[0,\tfrac{1}{n}]$
  and in $[1-\tfrac1n,1]$.
  In addition, suppose that $(\psi_n)_{n\in\Nat}$
  approximates $\psi$ uniformly in $[0,1]$ and that $(D_\mt D_\St \psi_n)_{n\in\Nat}$ 
  approximates $D_\mt D_\St \psi$ pointwise and boundedly in $(0,1)$.
  Note that
  \begin{equation}
    D_\St \psi_n\rob{\frac{1}{n}-}=0\ \text{ and }\ 
    D_\St \psi_n\rob{\frac{1}{n}+}=D_\St\psi\rob{\frac{1}{n}}.
  \end{equation}
  Comparing the occupation time
  formula~\eqref{eq:occupation_time_formula_Lb}
  of $\Lb_p(\cdot)$
  with the occupation time
  formula~\eqref{eq:occupation_time_formula_Lt}
  of $\Lt_p(\cdot)$
  we see that
  \begin{equation}  \label{eq:Lt_Lb}
    \Lb_p(t)=\Lt_p(t)\frac{1}{2} v(p)\mt(p)
    =\Lt_p(t)\frac{1}{2\St^{'}(p)}\quad\fa p\in(0,1).
  \end  {equation}
  Now applying the
  It\^o-Tanaka formula~\eqref{eq:ito_tanaka_semimartingale}
  to $\psi_n(\cdot)$ and
  inserting~\eqref{eq:Lt_Lb}, we arrive at
  \begin{equation}  \begin{split}   \label{eq:ito_tanaka_n}
    \psi_n(\Yt_t)-\psi_n(\Yt_0)
    &=\int_0^t D_\mt D_\St\psi_n(\Yt_u)\,du
      +\int_0^t \psi_n^{'}(\Yt_u)\sqrt{v(\Yt_u)}dB_u\\
    &\quad
       +\frac{1}{2} \Lt_{\frac{1}{n}}(t) D_\St \psi_n\rob{ \frac1n }
       -\frac{1}{2} \Lt_{1-\frac{1}{n}}(t) D_\St \psi_n\rob{ 1-\frac1n }.
  \end  {split}     \end  {equation}
  Note that the Lebesgue measure of $\curlb{u\leq t\colon \Yt_u\in\{0,1\}}$
  is equal to zero almost surely.
  Letting $n\to\infty$ in~\eqref{eq:ito_tanaka_n} completes the proof.
  \qed
\end  {proof}


\begin{proof}[Proof of Theorem~\ref{thm:main_result}]
  Recall $\Gt$, $D_\mt D_\St$ and $\coreGent$ from
  Section~\ref{sec:the generator of the time-reversed process}.
  Lemma~\ref{l:generator_backward_process} shows that the generator of
  the time-reversed process is the closure of $\Gt$.
  Therefore it remains to be shown that the generator of the Markov process
  $\rob{\pt(t)\colon t\geq0}$ restricted to the
  set $\coreGent$ coincides with $\Gt$, that is, that
  \begin{equation}  \label{eq:conv_to_Gent}
    \frac{\Expectation^p \psi\rob{\pt(t)}-\psi(p)}{t}
    \lratO \Gt \psi(p)=D_\mt D_\St \psi(p)
  \end  {equation}
  holds for all $p\in[0,1]$ and every $\psi\in\coreGent$.

  Recall $R_i,r_i,\kappa_i$ for $i\in\{0,1\}$ from
  Section~\ref{sec:main_result}.
  Fix $\psi\in\coreGent$ and note that $\psi$ is $\C^2$ in $(0,1)$.
  Using It\^o's formula,
  it is straightforward to show that the convergence in~\eqref{eq:conv_to_Gent}
  holds for every $p\in(0,1)$ if $\lambda \kappa_i\neq 0$
  and holds for every $p\in[0,1]$ if $\lambda\kappa_i=0$.
  It remains to prove~\eqref{eq:conv_to_Gent} for $i\in\{0,1\}$
  if $\lambda\kappa_i\neq0$.
  Starting at $i\in\{0,1\}$, $\rob{\pt(t)\colon t\geq0}$ evolves according to
  a diffusion $\rob{\Yt(t)\colon t\geq0}$ which is governed by $\At$
  until the first time $t$ such that
  $\Lt_i(t)\geq R_i$.
  At that time, the process restarts from an independent random
  point $J_i$ in $(0,1)$ with distribution
  $\tfrac{1}{\kappa_i}w_i(p)\pi(p)\,dp$.
  Thus
  \begin{equation}  \begin{split} \label{eq:main_calc}
    \lefteqn{\Expectation^i \psi(\pt(t))-\psi(i)
    -\Expectation^i\eckB{ \1_{\Lt_i(t)<R_i}\rob{\psi(\Yt_t)-\psi(i)}}}\\
    &=\Expectation^i\eckB{\int_0^{\Lt_i(t)}
         \Expectation^{J_i}\eckb{ \psi(\pt\ro{t-l})-\psi(i)}
         r_i e^{-r_i l}\,dl}\\
    &=\Expectation^i\eckB{\int_0^{\Lt_i(t)}
      \ruB{\int_0^1 \rob{\psi(z)-\psi(i)}
        \frac{1}{\kappa_i}  w_i(z)\pi(z)\,dz+O(t-l)}
        r_i e^{-r_i l}\,dl}\\
    &=\Expectation^i\eckb{1-e^{-r_i \Lt_i(t)}}
      \int_0^1 \eckb{ \psi(z)-\psi(i)}
      \frac{1}{\kappa_i} w_i(z)\pi(z)\,dz
        +O\rob{t\mal\Expectation^i\Lt_i(t)}\\
    &=o(t)+\Expectation^i\eckb{r_i \Lt_i(t)}
        (-1)^{i+1} \lim_{p\to i}\rob{\tfrac12 v\pi\psi^{'}}(p)\frac{1}{\lambda\kappa_i}
        +O\rob{t\mal\Expectation^i\Lt_i(t)}\\
  \end  {split}     \end  {equation}
  as $t\to0$.
  In the last step we used the inequality $1-e^{-x}-x\leq x^2$ for $x\geq0$
  together with Lemma~\ref{l:Lt_is_o_of_sqrt_t} and $\psi\in\coreGent$.
  The local time $\Lt_i(t)$ converges to zero a.s.\ as $t\to0$.
  By the dominated convergence theorem, this implies 
  that $\Expectation^i\Lt_i(t)$ converges to zero as $t\to0$.
  Thus the last
  summand on the right-hand side of~\eqref{eq:main_calc} is of order $o(t)$.
  Furthermore Lemma~\ref{l:rate} implies
  \begin{equation}   \label{eq:first_calc}
     r_i \lim_{p\to i}\rob{v\pi\psi^{'}}(p)\frac{1}{\lambda\kappa_i}
     =\lim_{p\to i}\rob{v\mt\psi^{'}}(p)=D_\St \psi(i).
  \end{equation}
  Next we consider the second expectation on the left-hand side
  of~\eqref{eq:main_calc}.
  Using H\"older's inequality we see that
  \begin{equation}  \begin{split}
    \Expectation^i&\eckB{ \1_{\Lt_i(t)\geq R_i}\rob{\psi(\Yt(t))-\psi(i)}}
    =\Expectation^i\eckB{ \rob{1-e^{-r_i \Lt_i(t)}}
                           \rob{\psi(\Yt(t))-\psi(i)}
                         }\\
    &\leq\sqrt{\Expectation^i\eckB{ \rob{\Lt_i(t)}^2} }
         \sqrt{\Expectation^i\eckB{ \rob{\psi(\Yt(t))-\psi(i)}^2} }
    =\sqrt{o(t)}\sqrt{O(t)}=o(t)
  \end{split}     \end{equation}
  where we have applied Lemma~\ref{l:Lt_is_o_of_sqrt_t}.
  Thus  we obtain from
  the It\^o-Tanaka formula~\eqref{eq:Ito_Tanaka_Y_tilde}
  \begin{equation}  \begin{split}  \label{eq:second_calc}
    &\Expectation^i\eckB{ \1_{\Lt_i(t)< R_i}\rob{\psi(\Yt(t))-\psi(i)}}
    =o(t)+ \Expectation^i{ \eckb{\psi(\Yt(t))-\psi(i)} }\\
    &= o(t)+t D_\mt D_\St \psi(i) 
       +(-1)^i \frac{1}{2}\Expectation^i\eckb{\Lt_{i}(t)}\mal D_{\St}\psi(i)
  \end{split}     \end{equation}
  as $t\to0$.
  Putting~\eqref{eq:main_calc}, \eqref{eq:first_calc}
  and~\eqref{eq:second_calc} together
  completes the proof of Theorem~\ref{thm:main_result}.
  \qed
\end  {proof}

\acks
We are grateful to Tom Kurtz, Alison Etheridge and two anonymous referees for 
their suggestions and comments on the manuscript.

\hyphenation{Sprin-ger}

\end{document}